\documentclass[12pt,british]{article}
\usepackage{amsmath,amsfonts,amssymb,amscd,amsthm}
\usepackage{babel}
\usepackage{enumerate,url,tikz-cd}
\usepackage{verbatim}
\usepackage[matrix,arrow]{xy}

\usepackage[pdftex,%
bookmarks=true,
backref=page,
breaklinks=true,
bookmarksnumbered = true,
colorlinks= true,
urlcolor= green,
anchorcolor = yellow,
citecolor=blue,
]{hyperref}                   

\usepackage{tikz}
\usepackage{tikz-cd}
\usetikzlibrary{arrows,decorations.markings,calc}
\tikzset{commutative diagrams/.cd,arrow style=tikz,diagrams={>=stealth'}}
\tikzset{middlearrow/.style={decoration={markings,mark=at position #1 with {\arrow[very thick]{stealth' reversed}}},postaction={decorate}}}
\tikzset{middlearrowrev/.style={decoration={markings,mark=at position #1 with {\arrow[very thick]{stealth'}}},postaction={decorate}}}
\usepackage{pgfplots}

\def\centerarc[#1](#2)(#3:#4:#5)
{\draw[#1] ($(#2)+({#5*cos(#3)},{#5*sin(#3)})$) arc (#3:#4:#5);}

\renewcommand{\epsilon}{\ensuremath{\varepsilon}}
\renewcommand{\to}{\ensuremath{\longrightarrow}}
\newcommand{\R}{\ensuremath{\mathbb R}}
\newcommand{\T}{\ensuremath{\mathbb{T}^{2}}}
\newcommand{\rp}{\ensuremath{{\mathbb R}P^2}}
\newcommand{\C}{\ensuremath{\mathbb C}}
\newcommand{\N}{\ensuremath{\mathbb N}}
\newcommand{\Z}{\ensuremath{\mathbb Z}}
\newcommand{\dt}[1][2]{\ensuremath{\mathbb D}^{#1}}
\newcommand{\St}[1][2]{\ensuremath{\mathbb S}^{#1}}
\newcommand{\vide}{\ensuremath{\varnothing}}
\newcommand{\im}[1]{\ensuremath{\operatorname{\text{Im}}\left({#1}\right)}}
\renewcommand{\ker}[1]{\ensuremath{\operatorname{\text{Ker}}\left({#1}\right)}}

\makeatletter
\def\@enum@{\list{\csname label\@enumctr\endcsname}%
           {\usecounter{\@enumctr}\def\makelabel##1{
\normalfont\ignorespaces\emph{{##1}~}}
\setlength{\labelsep}{3pt}
\setlength{\parsep}{0pt}
\setlength{\itemsep}{0pt}
\setlength{\leftmargin}{0pt}
\setlength{\labelwidth}{0pt}
\setlength{\listparindent}{\parindent}%
\setlength{\itemsep}{0pt}
\setlength{\itemindent}{0pt}
\setlength{\topsep}{3pt plus 1pt minus 1 pt}}}
\renewcommand\theenumi{\@alph\c@enumi}
\renewcommand\theenumii{\@alph\c@enumii}
\renewcommand\theenumiii{\@alph\c@enumiii}
\renewcommand\theenumiv{\@alph\c@enumiv}

\makeatother

\newcommand{\map}[4][\to]{\ensuremath{{#2} \colon\thinspace {#3} #1 {#4}}}
\newcommand{\id}{\ensuremath{\operatorname{\text{Id}}}}
\newcommand{\lhra}{\mathrel{\lhook\joinrel\to}}

\newcommand{\splitmap}[3]{\operatorname{\text{Sp}}(#1,#2,#3)}

\DeclareRobustCommand*{\up}[1]{\textsuperscript{#1}}
\renewcommand{\th}{\ensuremath{\up{th}}}
\newcommand{\ft}[1][n]{\ensuremath{\Delta_{#1}^{2}}}

\newcommand{\brak}[1]{\ensuremath{\left\{ #1 \right\}}}
\newcommand{\ang}[1]{\ensuremath{\left\langle #1\right\rangle}}
\newcommand{\set}[2]{\ensuremath{\brak{#1 \,\mid\, #2}}}

\newcommand{\setr}[2]{\ensuremath{\brak{#1 \,\left\lvert \, #2 \right.}}}

\newtheoremstyle{theoremm}{}{}{\itshape}{}{\scshape}{.}{ }{}

\theoremstyle{theoremm}
\newtheorem{thm}{Theorem}
\newtheorem{lem}[thm]{Lemma}
\newtheorem{prop}[thm]{Proposition}
\newtheorem{cor}[thm]{Corollary}

\newtheoremstyle{remarkk}{}{}{}{}{\scshape}{.}{ }{}

\theoremstyle{remarkk}
\newtheorem*{defn}{Definition}
\newtheorem*{defns}{Definitions}
\newtheorem*{rem}{Remark}
\newtheorem{rems}[thm]{Remarks}

\newcommand{\reth}[1]{Theorem~\protect\ref{th:#1}}
\newcommand{\relem}[1]{Lemma~\protect\ref{lem:#1}}
\newcommand{\repr}[1]{Proposition~\protect\ref{prop:#1}}
\newcommand{\reco}[1]{Corollary~\protect\ref{cor:#1}}
\newcommand{\resec}[1]{Section~\protect\ref{sec:#1}}
\newcommand{\req}[1]{equation~(\protect\ref{eq:#1})}
\newcommand{\reqref}[1]{(\protect\ref{eq:#1})}

\begin{document}

\title{Fixed points of $n$-valued maps, the fixed point property and the case of surfaces~--~a braid approach}

\author{DACIBERG~LIMA~GON\c{C}ALVES\\
Departamento de Matem\'atica - IME-USP,\\
Caixa Postal~66281~-~Ag.~Cidade de S\~ao Paulo,\\ 
CEP:~05314-970 - S\~ao Paulo - SP - Brazil.\\
e-mail:~\url{dlgoncal@ime.usp.br}\vspace*{4mm}\\
JOHN~GUASCHI\\
Normandie Universit\'e, UNICAEN,\\
Laboratoire de Math\'ematiques Nicolas Oresme UMR CNRS~\textup{6139},\\
CS 14032, 14032 Caen Cedex 5, France.\\
e-mail:~\url{john.guaschi@unicaen.fr}}

\maketitle

\begin{abstract}
\noindent We study the fixed point theory of $n$-valued maps of a space $X$ using the fixed point theory of maps between $X$ and its configuration spaces. We give some general results to decide whether an $n$-valued map can be deformed to a fixed point free $n$-valued map. In the case of surfaces, we provide an algebraic criterion in terms of the braid groups of $X$ to study this problem. If $X$ is either the $k$-dimensional ball or an even-dimensional real or complex projective space, we show that the fixed point property holds for $n$-valued maps for all $n\geq 1$, and we prove the same result for even-dimensional spheres for all $n\geq 2$. If $X$ is the $2$-torus, we classify the homotopy classes of $2$-valued maps in terms of the braid groups of $X$. We do not currently have a complete characterisation of the homotopy classes of split $2$-valued maps of the $2$-torus that contain a fixed point free representative, but we give an infinite family of such homotopy classes.
%
\end{abstract}

\section{Introduction}\label{sec:intro}

Multifunctions and their fixed point theory have been studied for a number of years, see for example the books~\cite{Be,Gor}, where fairly general classes of multifunctions and spaces are considered. Continuous $n$-valued functions are of particular interest, and more information about their fixed point theory on finite complexes may be found in~\cite{Brr1,Brr2,Brr3,Brr4,Brr5,Brr6,Bet1,Bet2,Sch0,Sch1,Sch2}. In this paper, we will concentrate our attention on the case of metric spaces, and notably that of surfaces. In all of what follows, $X$ and $Y$ will be topological spaces, and  $\phi\colon\thinspace  X \multimap  Y$ will be a multivalued function \emph{i.e.}\  a function that to each $x\in X$ associates a non-empty subset $\phi(x)$ of $Y$. Following the notation and the terminology of the above-mentioned papers, a multifunction $\phi\colon\thinspace  X \multimap  Y$  is \emph{upper semi-continuous} if for all $x\in X$, $\phi(x)$ is closed, and given an open set $V$ in $Y$, the set $\set{x\in X}{\phi(x)\subset V}$ is open in $X$, is \emph{lower semi-continuous} if the set $\set{x \in  X}{\phi(x)\cap V \neq  \varnothing}$ is open in $X$, and is \emph{continuous} if it is upper semi-continuous and lower semi-continuous. Let $I$ denote the unit interval $[0,1]$. We recall the definitions of (split) $n$-valued maps.

\begin{defns}
Let $X$ and $Y$ be topological spaces, and let $n\in \N$.
\begin{enumerate}
\item An \emph{$n$-valued map} (or multimap) $\phi\colon\thinspace X \multimap Y$,  is a continuous multifunction that to each $x\in X$ assigns an unordered subset of $Y$ of cardinal exactly $n$.  
\item A \emph{homotopy} between two $n$-valued maps $\phi_1,\phi_2\colon\thinspace X \multimap Y$ is an $n$-valued  map $H\colon\thinspace  X\times I \multimap Y$ such that $\phi_1=H ( \cdot , 0)$ and  $\phi_2=H ( \cdot , 1)$.
\end{enumerate}
\end{defns}

\begin{defn}[\cite{Sch0}]
An $n$-valued function $\phi\colon\thinspace  X \multimap Y$ is said to be a \emph{split $n$-valued map} if there exist single-valued maps $f_1, f_2, \ldots, f_n\colon\thinspace X \to Y$ such that $\phi(x)=\brak{f_1(x),\ldots,f_n(x)}$ for all $x\in X$. This being the case, we shall write $\phi=\brak{f_1,\ldots,f_n}$. Let $\splitmap{X}{Y}{n}$ denote the set of split $n$-valued maps between $X$ and $Y$.
 \end{defn}

\emph{A priori}, $\phi\colon\thinspace X \multimap Y$ is just an $n$-valued function, but if it is split then it is continuous by \repr{multicont} in the Appendix, which justifies the use of the word `map' in the above definition. Partly for this reason, split $n$-valued maps play an important r\^ole in the theory.

We now  recall  the notion of coincidence of a pair $(\phi, f)$ where $\phi$ is an $n$-valued map and  $f\colon\thinspace X \to Y$ is a  single-valued map (meaning continuous)~\emph{cf.}~\cite{Brr5}.  Let $\id_{X}\colon\thinspace X \to X$ denote the identity map of $X$. 

\begin{defns}
Let $\phi\colon\thinspace  X\multimap Y$  be an $n$-valued map, and let $f\colon\thinspace X \to Y$ be a single-valued map. The set of coincidences of the pair $(\phi, f)$ is denoted by $\operatorname{\text{Coin}}(\phi, f)=\set{x\in X}{f(x)\in \phi(x)}$.  If $X=Y$ and $f=\id_X$ then $\operatorname{\text{Coin}}(\phi, \id_X)=\set{x\in X}{x\in \phi(x)}$ is called the \emph{fixed point set} of $\phi$, and will be denoted by $\operatorname{\text{Fix}}(\phi)$. If $f$ is the constant map $c_{y_0}$ at a point $y_0\in Y$ then $\operatorname{\text{Coin}}(\phi, c_{y_0})=\set{x\in X}{y_0\in \phi(x)}$ is called the set of \emph{roots} of $\phi$ at $y_0$. 
\end{defns} 

Recall that a space $X$ is said to have the \emph{fixed point property} if any self-map of $X$ has a fixed point. This notion may be generalised to $n$-valued maps as follows.

\begin{defn}\label{fppro} 
If $n\in \N$, a space $X$ is said to have the \emph{fixed point property} for $n$-valued maps if any $n$-valued map  $\phi\colon\thinspace  X \multimap X$ has a fixed point.
\end{defn}
 
If $n=1$ then we obtain the classical notion of the fixed point property. It is well known that the fixed point theory of surfaces is more complicated than that of manifolds of higher dimension. This is also the case for $n$-valued maps. A number of results for singled-valued maps of manifolds of dimension at least three may be generalised to the setting of $n$-valued maps, see for example the results of Schirmer from the 1980's~\cite{Sch0,Sch1,Sch2}. In dimension one or two, the situation is more complex, and has only been analysed within the last ten years or so, see~\cite{Brr1} for the study of $n$-valued maps of the circle. The papers~\cite{Brr4,Brr6} illustrate some of the difficulties that occur when the manifold is the $2$-torus $\T$. Our expectation is that the case of surfaces of negative Euler characteristic will be much more involved.

In this paper, we explore the fixed point property for $n$-valued maps, and we extend the famous result of L.~E.~J.~Brouwer that every self-map of the disc has a fixed point to this setting  \cite{Bru}. We will also develop some tools to decide whether an $n$-valued map can be deformed to a fixed point free $n$-valued map, and we give a partial classification of those split $2$-valued maps of $\T$ that can be deformed to fixed point free $2$-valued maps. Our approach to the study of fixed point theory of $n$-valued maps makes use of the homotopy theory of configuration spaces. It is probable that these ideas can also be adapted to coincidence theory. This viewpoint is fairly general. It helps us to understand the theory, and provides some means to perform (not necessarily easy) computations in general. Nevertheless, for some specific situations, such as for surfaces of non-negative Euler characteristic, these calculations are often tractable. To explain our approach, let $F_{n}(Y)$ denote the \emph{$n\th$ (ordered) configuration space} of a space $Y$, defined by:
\begin{equation*}
F_n(Y)=\setr{(y_1,\ldots,y_n)}{\text{$y_i\in Y$, and $y_i\neq y_j$ if $i\neq j$}}.
\end{equation*}
Configuration spaces play an important r\^ole in several branches of mathematics and have been extensively studied, see~\cite{CG,FH} for example. The symmetric group $S_n$ on $n$ elements acts freely on $F_n(Y)$ by permuting coordinates. The corresponding quotient space, known as the \emph{$n\th$ (unordered) configuration space of $Y$}, will be denoted by $D_n(Y)$, and the quotient map will be denoted by $\pi \colon\thinspace  F_{n}(Y) \to D_{n}(Y)$. The \emph{$n\th$ pure braid group $P_n(Y)$} (respectively the \emph{$n\th$ braid group $B_n(Y)$}) of $Y$ is defined to be the fundamental group of $F_n(Y)$ (resp.\ of $D_n(Y)$), and there is a short exact sequence:
\begin{equation}\label{eq:sesbraid}
1\to P_n(Y) \to B_n(Y) \stackrel{\tau}{\to} S_n \to 1,
\end{equation}
where $\tau$ is the homomorphism that to a braid associates its induced permutation. For $i=1,\ldots,n$, let $p_{i}\colon\thinspace F_{n}(Y)\to Y$ denote projection onto the $i\th$ factor. The notion of intermediate configuration spaces was defined in~\cite{GG2,GG4}. More precisely, if $n, m\in \N$,  the subgroup $S_n\times S_m \subset S_{n+m}$ acts freely on $F_{n+m}(Y)$ by restriction, and the corresponding orbit space $F_{n+m}(Y)/(S_n\times S_m)$ is denoted by $D_{n, m}(Y)$. Let $B_{n,m}=\pi_{1}(D_{n, m}(Y))$ denote the associated `mixed' braid group. The space $F_{n+m}(Y)$ is equipped with the topology induced by the inclusion $F_{n+m}(Y)\subset Y^{n+m}$, and $D_{n, m}(Y)$ is equipped with the quotient topology. If $Y$ is a manifold without boundary then the natural projections $\overline{p}_{m,n}\colon\thinspace  D_{m,n}(Y) \to D_m(Y)$ onto the first $m$ coordinates are fibrations. For maps whose target is a configuration space, we have the following notions.

\begin{defns}
Let $X$ and $Y$ be topological spaces, and let $n\in \N$. A map $\Phi\colon\thinspace X \to D_n(Y)$ will be called an \emph{$n$-unordered map}, and a map $\Psi\colon\thinspace X \to F_n(Y)$ will be called an \emph{$n$-ordered map}. For such an $n$-ordered map, for $i=1,\ldots,n$, there exist maps $f_i\colon\thinspace X \to Y$ such that $\Psi(x)=(f_1(x),\ldots, f_n(x))$ for all $x\in X$, and for which $f_i(x)\neq f_j(x)$ for all $1\leq i,j\leq n$, $i\neq j$, and all $x\in X$. In this case, we will often write $\Psi=(f_{1},\ldots,f_{n})$.  
\end{defns}

The fixed point-theoretic concepts that were defined earlier for $n$-valued maps carry over naturally to $n$-unordered  and $n$-ordered maps as follows.

\begin{defns}
Let $X$ and $Y$ be topological spaces, let $f\colon\thinspace X \to Y$ be a single-valued map, let $y_0\in Y$, and let $n\in \N$. 
\begin{enumerate}[(a)]
\item Given an $n$-unordered map $\Phi\colon\thinspace X \to D_n(Y)$, $x\in X$ is said to be a \emph{coincidence} of the pair $(\Phi, f)$ if there exist $(x_1,\ldots,x_n)\in F_n(Y)$ and $j\in\brak{1,\ldots,n}$ such that $\Phi(x)= \pi(x_1,\ldots, x_n)$ and $f(x)=x_j$. The set of coincidences of the pair $(\Phi, f)$ will be denoted by $\operatorname{\text{Coin}}(\Phi, f)$.
If $X=Y$ and $f=\id_X$ then $\operatorname{\text{Coin}}(\Phi, \id_X)$
is called the \emph{fixed point set} of $\Phi$, and is denoted by  $\operatorname{\text{Fix}}(\Phi)$. If $f$ is the constant map $c_{y_0}$ at $y_0$ then $\operatorname{\text{Coin}}(\Phi, c_{y_0})$
is called the set of \emph{roots} of $\Phi$ at $y_0$.

\item Given an $n$-ordered  map $\Psi\colon\thinspace  X\to F_n(Y)$, the set of coincidences of the pair $(\Psi, f)$ is defined by $\operatorname{\text{Coin}}(\Psi, f)= \set{x\in X}{\text{$f(x)= p_j\circ \Psi(x)$ for some $1\leq j\leq n$}}$.  If $X=Y$ and $f=\id_X$ then $\operatorname{\text{Coin}}(\Psi, \id_X)=\set{x\in X}{\text{$x=p_j\circ \Psi(x)$ for some $1\leq j\leq n$}}$  is called the \emph{fixed point set} of $\Psi$, and is denoted by  $\operatorname{\text{Fix}}(\Psi)$. If $f$ is the constant map $c_{y_0}$ then $\operatorname{\text{Coin}}(\Psi, c_{y_0})=\set{x\in X}{\text{$y_0=p_j\circ \Psi(x)$ for some $1\leq j\leq n$}}$ is called the set of \emph{roots} of $\Psi$ at $y_0$. 
\end{enumerate}  
\end{defns} 

In order to study $n$-valued maps via single-valued maps, we use the following natural relation between multifunctions and functions. First observe that there is an obvious bijection between the set of $n$-point subsets of a space $Y$ and the unordered configuration space $D_{n}(Y)$. This bijection induces a one-to-one correspondence between the set of $n$-valued functions from $X$ to $Y$ and the set of functions from $X$ to $D_n(Y)$. In what follows, given an $n$-valued function $\phi\colon\thinspace  X \multimap Y$, we will denote the corresponding function whose target is the configuration space $D_{n}(Y)$ by $\Phi\colon\thinspace  X \to D_n(Y)$, and \emph{vice-versa}. Since we are concerned with the study of continuous multivalued functions, we wish to ensure that this correspondence restricts to a bijection between the set of (continuous) $n$-valued maps and the set of continuous single-valued maps whose target is $D_{n}(Y)$. It follows from \reth{metriccont} that this is indeed the case if $X$ and $Y$ are metric spaces. This hypothesis will clearly be satisfied throughout this paper. If the map $\Phi\colon\thinspace  X \to D_n(Y)$ associated to $\phi$ admits a lift $\widehat{\Phi}\colon\thinspace  X \to F_n(Y)$ via the covering map $\pi$ then we shall say that $\widehat{\Phi}$ is a \emph{lift} of $\phi$ (see \resec{pres} for a formal statement of this definition). We will make use of this notion to develop a correspondence between split $n$-valued maps and maps from $X$ into $F_n(Y)$. As we shall see, the problems that we are interested in for $n$-valued maps, such as coincidence, fixed point and root problems, may be expressed within the context of $n$-unordered maps, to which we may apply the classical theory of single-valued maps.  

Our main aims in this paper are to explore the fixed point property of spaces for $n$-valued maps, and to study the problem of whether an $n$-valued map map can be deformed to a fixed point free $n$-valued map. We now give the statements of the main results of this paper. The first theorem shows that for simply-connected metric spaces, the usual fixed point property implies the fixed point property for $n$-valued maps.
 
\begin{thm}\label{th:sccfpp}
Let $X$ be a simply-connected metric space that has the fixed point property, and let $n\in \N$. Then every $n$-valued map of $X$ has at least $n$ fixed points, so $X$ has the fixed point property for $n$-valued maps. In particular, for all $n,k\geq 1$, the $k$-dimensional disc $\dt[k]$ and the $2k$-dimensional complex projective space $\C P^{2k}$  have the fixed point property for $n$-valued maps.
\end{thm}

It may happen that a space does not have the (usual) fixed point property but that it has the fixed point property for $n$-valued maps for $n>1$. This is indeed the case for the $2k$-dimensional sphere $\St[2k]$.
 
\begin{prop}\label{prop:S2fp}
If $n\geq 2$ and $k\geq 1$, $\St[2k]$ has the fixed point property for $n$-valued maps.
\end{prop}

\reth{sccfpp} and \repr{S2fp} will be proved in \resec{pres}. Although the $2k$-dimensional real projective space $\R P^{2k}$ is not simply connected, in \resec{rp2k} we will show that it has the fixed point property for $n$-valued maps for all $n\in \N$.

\begin{thm}\label{th:rp2Kfpp}
Let $k,n\geq 1$. The real projective space $\R P^{2k}$ has the fixed point property for $n$-valued maps. Further, any $n$-valued map of $\R P^{2k}$ has at least $n$ fixed points.
\end{thm}

We do not know of an example of a space that has the fixed point property, but that does not have the fixed point property for $n$-valued maps for some $n\geq 2$.

In \resec{fixfree}, we turn our attention to the question of deciding whether an $n$-valued map of a surface $X$ of non-negative Euler characteristic $\chi(X)$ can be deformed to a fixed point free $n$-valued map. In the following result, we give algebraic criteria involving the braid groups of $X$.

\begin{thm}\label{th:defchineg}
Let $X$ be a compact surface without boundary such that $\chi(X)\leq 0$, let $n\geq 1$, and let $\phi\colon\thinspace  X \multimap X$ be an $n$-valued map.
\begin{enumerate}[(a)]
\item\label{it:defchinega} The $n$-valued map $\phi$ can be deformed to a fixed point free $n$-valued  map if and only if there is a homomorphism 
$\varphi\colon\thinspace  \pi_1(X) \to B_{1,n}(X)$ that makes the following diagram commute:
\begin{equation}\label{eq:commdiag1}
\begin{tikzcd}[ampersand replacement=\&]
\&\& B_{1,n}(X) \ar{d}{(\iota_{1,n})_{\#}}\\
\pi_{1}(X) \ar[swap]{rr}{(\id_{X}\times \Phi)_{\#}}  \ar[dashrightarrow, end anchor=south west]{rru}{\varphi}  \&\&\pi_{1}(X) \times B_{n}(X),
\end{tikzcd}
\end{equation}
where $\iota_{1,n}\colon\thinspace D_{1,n}(X) \to X \times D_{n}(X)$ is the inclusion map. 

\item\label{it:defchinegb} If the $n$-valued map $\phi$ is split, it can be deformed to a fixed point free $n$-valued  map if and only if there is a homomorphism $\widehat{\varphi}\colon\thinspace  \pi_1(X) \to P_{n+1}(X)$ that makes the following diagram commute:
\begin{equation*}
\begin{tikzcd}[ampersand replacement=\&]
\&\& P_{n+1}(X) \ar{d}{(\widehat{\iota}_{n+1})_{\#}}\\
\pi_{1}(X) \ar[swap]{rr}{(\id_{X}\times \widehat{\Phi})_{\#}}  \ar[dashrightarrow, end anchor=south west]{rru}{\widehat{\varphi}}  \&\& \pi_{1}(X) \times P_{n}(X).
\end{tikzcd}
\end{equation*} 
where $\widehat{\iota}_{n+1}\colon\thinspace F_{n+1}(X) \to X \times F_{n}(X)$ is the inclusion map.
\end{enumerate}\end{thm}

If $\phi\colon\thinspace  X \multimap X$ is a split $n$-valued map given by $\phi=\{f_1, \cdots ,f_n\}$ that can be deformed to a fixed point free $n$-valued map, then certainly each of the single-valued maps $f_i$ can be deformed to a fixed point free map.  The question of whether the converse of this statement holds for surfaces is open. We do not know the answer for any compact surface without boundary different from $\St$ or $\rp$, but it is likely that the converse does not hold. More generally, one would like to know if the homotopy class of $\phi$ contains a representative for which the number of fixed points is exactly the Nielsen number. Very little is known about this question, even for the $2$-torus.  Recall that the Nielsen number of an $n$-valued map $\phi\colon\thinspace  X\multimap X$, denoted $N(\phi)$, was defined by Schirmer~\cite{Sch1}, and generalises the usual Nielsen number in the single-valued case. She showed that $N(\phi)$ is a lower bound for the number of fixed points among all $n$-valued maps homotopic to $\phi$. 

Within the framework of \reth{defchineg}, it is natural to study first the case of $2$-valued maps of the $2$-torus $\T$, which is the focus of  \resec{toro}. In what follows, $\mu$ and $\lambda$ will denote the meridian and the longitude respectively of $\T$. Let $(e_1, e_2)$ be a basis of $\pi_1(\T)$ such that $e_1=[\mu]$ and $e_2=[\lambda]$. For self-maps of $\T$, we will not be overly concerned with the choice of basepoints since the fundamental groups of $\T$ with respect to two different basepoints may be canonically identified. In \resec{toro2}, we will study the groups $P_{2}(\T)$, $B_{2}(\T)$ and $P_{2}(\T\setminus\brak{1})$, and in \reco{compactpres}, we will see that $P_{2}(\T)$ is isomorphic to the direct product of a free group $\mathbb{F}_2(u,v)$ of rank $2$ and $\Z^2$. In what follows, the elements of $P_2(\T)$ will be written with respect to the decomposition $\mathbb{F}_2(u,v) \times \Z^2$, and $\operatorname{\text{Ab}}\colon\thinspace \mathbb{F}_2(u,v) \to \Z^{2}$ will denote Abelianisation. \reth{helgath01}, which is a result of~\cite{Sch1} for the Nielsen number of split $n$-valued maps, will be used in part of the proof of the following proposition.

\begin{prop}\label{prop:exisfpf}
Let $\phi\colon\thinspace \T \multimap \T$ be a split $2$-valued map of the torus $\T$, and let $\widehat{\Phi}=(f_1,f_2) \colon\thinspace \T \to F_{2}(\T)$ be a lift of $\phi$ such that $\widehat{\Phi}_{\#}(e_{1})=(w^r,(a,b))$ and $\widehat{\Phi}_{\#}(e_{2})= (w^s, (c,d)))$, where $(r,s)\in \Z^{2}\setminus \brak{(0,0)}$, $a,b,c,d\in \Z$ and $w\in \mathbb{F}_2(u,v)$. Then the Nielsen number of $\phi$ is given by:
\begin{equation*}
N(\phi)=\left\lvert\det\begin{pmatrix}
 a-1 & c  \\
 b & d-1 
\end{pmatrix}\right\rvert
+
\left\lvert\det\begin{pmatrix}
 rm+a-1 & sm+c  \\
 rn+b & sn+d-1 
\end{pmatrix}\right\rvert,
\end{equation*}
where $\operatorname{\text{Ab}}(w)=(m,n)\in \Z^{2}$.  If the map $\phi$ can be deformed to a fixed point free $2$-valued map, then both of the maps $f_1$ and $f_2$ can be deformed  to fixed point free maps. Furthermore, $f_1$ and $f_2$  can be deformed to fixed point free maps if and only if either:
\begin{enumerate}[(a)]
\item\label{it:exisfpfa} the pairs of integers $(a-1, b),(c,d-1)$ and $(m,n)$ belong to a cyclic subgroup of $\Z^2$, or
\item\label{it:exisfpfb} $s(a-1, b)=r(c,d-1)$.
\end{enumerate}
\end{prop}

Within the framework of \repr{exisfpf}, given a split $2$-valued map $\phi\colon\thinspace \T \multimap \T$ for which $N(\phi)=0$, we would like to know whether $\phi$ can be deformed to a fixed point free $2$-valued map. If $N(\phi)=0$, then by this proposition, one of the conditions~(\ref{it:exisfpfa}) or~(\ref{it:exisfpfb}) must be satisfied. The following result shows that condition~(\ref{it:exisfpfb}) is also sufficient.

\begin{thm}\label{th:necrootfree3} 
Let $\widehat{\Phi}\colon\thinspace \T \to F_2({\T})$ be a lift of a split $2$-valued map $\phi\colon\thinspace \T \multimap \T$ that satisfies $\widehat{\Phi}_{\#}(e_{1})=(w^{r},(a,b))$ and $\widehat{\Phi}_{\#}(e_{2})= (w^{s}, (c,d))$, where $w\in \mathbb{F}_2(u,v)$, $a,b,c,d\in \Z$ and $(r,s)\in \Z^{2}\setminus \brak{(0,0)}$ satisfy $s(a-1,b)=r(c,d-1)$. Then $\phi$ may be deformed to a fixed point free  $2$-valued map. 
 \end{thm}
 
With respect to condition~(\ref{it:exisfpfa}), we obtain a partial converse for certain values of $a,b,c,d,m$ and $n$.

\begin{thm}\label{th:construct2val}   Suppose that $(a-1, b),(c,d-1)$ and $(m,n)$  belong to a cyclic  subgroup of $\Z^2$ generated by an element of the form $(0,q), (1,q), (p,0)$ or $(p,1)$, where $p,q\in \Z$, and let $r,s\in \Z$. Then there exist $w\in \mathbb{F}_2(u,v)$, a split fixed point free $2$-valued map $\phi\colon\thinspace \T \multimap \T$ and a lift $\widehat{\Phi} \colon\thinspace  \T \to F_2(\T)$ of $\phi$ such that  such that $\operatorname{\text{Ab}}(w)=(m,n)$, $\widehat{\Phi}_{\#}(e_1)=((w^{r},(a,b))$ and $\widehat{\Phi}_{\#}(e_2)= (w^{s}, (c,d))$. 
\end{thm}

\repr{exisfpf} and Theorems~\ref{th:necrootfree3} and~\ref{th:construct2val} will be proved in \resec{toro}. Besides the introduction and an appendix, this paper is divided into 4 sections. In \resec{pres}, we give some basic definitions, we establish the connection between multimaps and maps whose target is a configuration space, and we show that simply-connected spaces have the fixed point property for $n$-valued maps if they have the usual fixed point property.  In \resec{rp2k}, we show that even-dimensional real projective spaces  have the fixed point property for $n$-valued maps. In \resec{fixfree}, we provide general criteria of a homotopic and algebraic nature, to decide whether an $n$-valued map can be deformed or not to a fixed point free $n$-valued map, and we give the corresponding statements for the case of roots.  In \resec{toro}, we study the fixed point theory of $2$-valued maps of the $2$-torus. In \resec{toro2}, we give presentations of certain braid groups of $\T$, in \resec{descript}, we describe the set of homotopy classes of split $2$-valued maps of $\T$, and in \resec{fptsplit2}, we study the fixed point theory of split $2$-valued maps. In the Appendix, written with R.~F.~Brown, in \reth{metriccont}, we show that for the class of metric spaces that includes those considered in this paper, $n$-valued maps can be regarded as single-valued maps whose target is the associated unordered configuration space.

\subsubsection*{Acknowledgements}

The authors are grateful to R.~F.~Brown for his collaboration in the writing of the Appendix. The first-named author would also like to thank him for interesting and useful discussions on the topics of this paper. The first-named author was partially supported by FAPESP-Funda\c c\~ao de Amparo a Pesquisa do Estado de S\~ao Paulo, Projeto Tem\'atico Topologia Alg\'ebrica, Geom\'etrica   2012/24454-8. The authors would like to thank 
the `R\'eseau Franco-Br\'esilien en Math\'ematiques' for financial support for their respective visits to the Laboratoire de Math\'ematiques Nicolas Oresme UMR CNRS~6139, Universit\'e de Caen Normandie, from the 2\textsuperscript{nd} to the 15\textsuperscript{th} of September 2016, and to the Instituto de Matem\'atica e Estat\'istica, Universidade de S\~ao Paulo, from the 9\textsuperscript{th} of July to the 1\textsuperscript{st} of August 2016. 

\section{Generalities and the $n$-valued fixed point property}\label{sec:pres}

In \resec{relnsnvm}, we begin by describing the relations between $n$-valued maps and $n$-unordered maps. We will assume throughout that $X$ and $Y$  are metric spaces, so that we can apply \reth{metriccont}. Making use of  unordered configuration space, in \relem{split} and \reco{inject}, we prove some properties about the fixed points of $n$-valued maps.  In  \resec{disc}, we give an algebraic condition that enables us to decide whether an $n$-valued map is split. We also study the case where $X$ is simply connected (the $k$-dimensional disc for example, which has the usual fixed point property) and we prove \reth{sccfpp}, and in \resec{sph}, we analyse the case of the $2k$-dimensional sphere (which does not have the usual fixed point property), and we prove \repr{S2fp}.
 
\subsection{Relations between $n$-valued maps, $n$-(un)ordered maps and their fixed point sets}\label{sec:relnsnvm}
 
A proof of the following result may be found in the Appendix. 

\begin{thm}\label{th:metriccont}
Let $X$ and $Y$ be metric spaces, and let $n\in \N$. An $n$-valued function $\phi\colon\thinspace X \multimap  Y$ is continuous if and only if the 
corresponding function $\Phi\colon\thinspace  X \to D_n(Y)$ is continuous.
\end{thm}

It would be beneficial for the statement of \reth{metriccont} to hold under weaker hypotheses on $X$ and $Y$.  See~\cite{Brr7} for some recent results  in this direction. 

\begin{defn} 
If $\phi\colon\thinspace  X \multimap Y$ is an $n$-valued map and $\Phi\colon\thinspace  X\to D_{n}(Y)$ is the associated $n$-unordered map, an $n$-ordered map $\widehat{\Phi}\colon\thinspace  X \to F_n(Y)$ is said to be a \emph{lift} of $\phi$ if the composition $\pi\circ \widehat{\Phi}\colon\thinspace X \to D_{n}(Y)$ of $\widehat{\Phi}$ with the covering map $\pi\colon\thinspace  F_n(Y) \to D_n(Y)$ is equal to $\Phi$.
\end{defn} 

If $\phi=\brak{f_1,\ldots,f_n}\colon\thinspace X \multimap Y$  is a split $n$-valued map, then it admits a lift $\widehat{\Phi}=(f_1,\ldots,f_n)\colon\thinspace X \to F_n(Y)$. For any such lift, $\operatorname{\text{Fix}}(\widehat{\Phi})=\operatorname{\text{Fix}}(\phi)$, and the map
$\widehat{\Phi}$ determines  an ordered set of $n$ maps $(f_1=p_1\circ \widehat{\Phi},\ldots, f_n=p_n\circ  \widehat{\Phi})$ from $X$ to $Y$ for which $f_i(x)\ne f_j(x)$ for all $x\in X$ and all $1\leq i<j\leq n$. Conversely, any ordered set of $n$ maps  $(f_1,\ldots,f_n)$ from $X$ to $Y$ for which $f_i(x)\ne f_j(x)$ for all $x\in X$ and all $1\leq i<j\leq n$ determines an $n$-ordered map $\Psi\colon\thinspace  X \to F_n(Y)$ defined by $\Psi(x)=(f_1(x),\ldots,f_n(x))$ and a split $n$-valued map $\phi=\brak{f_1,\ldots,f_n}\colon\thinspace X \multimap Y$  of which $\Psi$ is a lift. 
 So the existence of such a split $n$-valued map $\phi$ is equivalent to that of an $n$-ordered map $\Psi\colon\thinspace X \to F_n(Y)$, where $\Psi=(f_1,\ldots,f_n)$. This being the case, the composition $\pi\circ \Psi\colon\thinspace X \to D_n(Y)$ is the map $\Phi\colon\thinspace X \to D_n(Y)$ that corresponds (in the sense described in \resec{intro}) to the $n$-valued map $\phi$. Consequently, an $n$-valued  map  $\phi\colon\thinspace  X \multimap Y$ admits a lift if and only if it is split. As we shall now see, \reth{metriccont} will be of help in the description of the relations between (split) $n$-valued maps and $n$-(un)ordered maps of metric spaces. As we have seen, to each $n$-valued map (resp.\ split $n$-valued map),  we may associate an $n$-unordered map (resp.\ a lift), and \emph{vice-versa}. Note that the symmetric group $S_n$ not only acts (freely) on $F_n(Y)$ by permuting coordinates, but it also acts on the set of ordered $n$-tuples of maps between $X$ and $Y$. Further, the restriction of the latter action to the subset $F_{n}(Y)^{X}$ of $n$-ordered maps, \emph{i.e.}\ maps of the form $\Psi\colon\thinspace X \to F_{n}(Y)$, where $\Psi(x)=(f_1(x),\ldots,f_n(x))$ for all $x\in X$ for which $f_i(x)\ne f_j(x)$ for all $x\in X$ and $1\leq i< j\leq n$, is also free.  In what follows, $[X,Y]$ (resp.\ $[X,Y]_{0}$) will denote the set of homotopy classes (resp.\ based homotopy classes) of maps between $X$ and $Y$.

\begin{lem}\label{lem:split}\mbox{} Let $X$ and $Y$ be metric spaces, and let $n\in \N$. 
\begin{enumerate}[(a)]
\item\label{it:splitII} The set $\splitmap{X}{Y}{n}$ of split $n$-valued  maps from $X$ to $Y$ is in one-to-one correspondence with the orbits of the set of  maps $F_{n}(Y)^{X}$ from $X$ to $F_n(Y)$ modulo the free action defined above of $S_n$ on $F_{n}(Y)^{X}$.
\item\label{it:splitIII} If two $n$-valued maps from $X$ to $Y$ are homotopic and one is split, then the other is also split. Further, the set $\splitmap{X}{Y}{n}/\!\sim$ of homotopy classes of split $n$-valued  maps from $X$ to $Y$ is in one-to-one correspondence with the orbits of the set $[X,F_{n}(Y)]$ of homotopy classes of maps from $X$ to $F_n(Y)$ under the action of $S_n$ induced by that of $S_n$ on $F_{n}(Y)^{X}$.
\item\label{it:splitIV} Suppose that $X=Y$. If an  $n$-valued  map $\phi\colon\thinspace X \multimap X$ is split and deformable to a fixed point free map, then a lift 
$\widehat{\Phi}\colon\thinspace  X \to F_{n}(X)$  of $\phi$
may be written as ${\Psi}=(f_1,\ldots,f_n)$, where for all $i=1,\ldots,n$, the map $f_i\colon\thinspace X\to X$ is a map that is deformable to a fixed point free map.     
\end{enumerate}
\end{lem} 

\begin{proof}\mbox{}
\begin{enumerate}[(a)]
\item Let $\phi\colon\thinspace X \multimap Y$ be a  split $n$-valued map. From the definition, there exists an $n$-ordered map $\widehat{\Phi}\colon\thinspace X\to F_n(Y)$ such that $\Phi=\pi\circ \widehat{\Phi}$, up to the identification given by \reth{metriccont}.
If $\widehat{\Phi}=(f_1,\ldots,f_n)$, the other lifts of $\phi$ are obtained via the action of the group of deck  transformations of the covering space, this group being $S_n$ in our case, and so are of the form $(f_{\sigma(1)},\ldots,f_{\sigma(n)})$, where $\sigma\in S_{n}$. This gives rise to the stated one-to-one correspondence between $\splitmap{X}{Y}{n}$ and the orbit space $F_{n}(Y)^{X}/S_{n}$.

\item By naturality, the map $\pi\colon\thinspace  F_n(Y) \to D_n(Y)$ induces a map $\widehat{\pi}\colon\thinspace [X, F_n(Y)]\to [X, D_n(Y)]$ defined by $\widehat{\pi}([\Psi])=[\pi\circ \Psi]$ for any $n$-ordered  map $\Psi\colon\thinspace X\to F_{n}(Y)$. Given two homotopic $n$-valued maps between $X$ and $Y$, which we regard as maps from $X$ to $D_n(Y)$ using \reth{metriccont}, if the first has a lift to $F_n(Y)$, then the lifting property of a covering implies that the second also admits a lift to $F_n(Y)$, so if the first map is split then so is the second map. To prove the second part of the statement, first note that there is a surjective map $f\colon\thinspace  F_{n}(Y)^{X} \to \splitmap{X}{Y}{n}$ given by $f(g)=\pi\circ g$, 
where we identify $\splitmap{X}{Y}{n}$ with the set of maps $D_{n}(Y)^{X}$ from $X$ to $D_{n}(Y)$, that induces a surjective map $\overline{f}\colon\thinspace  [X,F_{n}(Y)] \to \splitmap{X}{Y}{n}/\!\sim$ on the corresponding sets of homotopy classes. Further, if ${\Psi}_1, {\Psi}_2\in F_{n}(Y)^{X}$ are two $n$-ordered maps that are homotopic via a homotopy $H$, and if $\alpha\in S_n$, then the maps $\alpha \circ\Psi_1,\alpha \circ \Psi_2 \in F_{n}(Y)^{X}$ are also homotopic via the homotopy $\alpha \circ H$, and so we obtain a quotient map $q\colon\thinspace  [X,F_{n}(Y)] \to [X,F_{n}(Y)]/S_{n}$. We claim that $\overline{f}$ factors through $q$ via the map $\overline{\overline{f}}\colon\thinspace  [X,F_{n}(Y)]/S_{n}\to \splitmap{X}{Y}{n}/\!\sim$ defined by $\overline{\overline{f}}([g])=[f(g)]$. To see this, let $g, h\in F_{n}(Y)^{X}$ be such that $q([g])=q([h])$. Then there exists $\alpha\in S_{n}$ such that $\alpha([g])=[h]$. Then $\overline{f}(\alpha[g])=[\alpha f(h)]= \overline{f}([h])=[f(h)]$. But from the definition of $\splitmap{X}{Y}{n}$, $[\alpha f(g)]=[f(g)]$, and so $[f(g)]=[f(h)]$, which proves the claim. By construction, the map $\overline{\overline{f}}$ is surjective. It remains to show that it is injective. Let $g,h\in F_{n}(Y)^{X}$ be such that $\overline{\overline{f}}([g])=\overline{\overline{f}}([h])$. Then $[f(g)]=[f(h)]$, and thus $f(g)$ and $f(h)$ are homotopic via a homotopy $H$ in $\splitmap{X}{Y}{n}$, where $H(0,f(g))=f(g)$ and $H(1,f(g))=f(h)$. Then $H$ lifts to a homotopy $\widetilde{H}$ such that $\widetilde{H}(0,g)=g$, and $\widetilde{H}(1,g)$ is a lift of $f(h)$. But $h$ is also a lift of $f(h)$, so there exists $\alpha\in S_{n}$ such that $f(h)=\alpha h$. Further, $g$ is homotopic to $f(h)$, so is homotopic to $\alpha(h)$, and hence $q([g])=q([\alpha(h)])= q([h])$ from the definition of $q$, which proves the injectivity of $\overline{\overline{f}}$.

\item Since $\phi$ is split, we may choose a lift $\widehat{\Phi}=(f_1,\ldots,f_n)\colon\thinspace  X \to F_{n}(X)$ of $\phi$. By hypothesis, there is a homotopy $H\colon\thinspace X\times I \to D_n(X)$ such that $H(\cdot ,0)=\Phi$, and $H(\cdot ,1)$ is fixed point free. Since the initial part of the homotopy $H$ admits a lift, there exists a lift $\widehat{H}\colon\thinspace X\times I \to F_n(X)$ of $H$ such that $\widehat{H}(\cdot, 0)=\widehat{\Phi}$, and $\widehat{H}(\cdot, 1)$ is fixed point free. So $\widehat{H}(\cdot, 1)$ is of the form $(f_1,\ldots,f_n)$,    where $f_i$ is fixed point free for $1\leq i\leq n$, 
and the conclusion follows.\qedhere
\end{enumerate}
\end{proof}
 
\begin{rems}\mbox{}
\begin{enumerate}[(a)]
\item The action of $S_n$ on the set of homotopy classes $[X, F_n(Y)]$ is not necessarily free (see \repr{classif}(\ref{it:classifb})). 

\item The question of whether the converse of \relem{split}(\ref{it:splitIV}) is valid for surfaces is open, see the introduction. 
\end{enumerate}
\end{rems}

The following consequence of \relem{split}(\ref{it:splitIV}) will be useful in what follows, and implies that if a split $n$-valued map can be deformed to a fixed point free $n$-valued map (through $n$-valued maps), then the deformation 
is through split $n$-valued maps.

\begin{cor}\label{cor:inject}  
Let $X$ be a metric space, and let $n\in \N$.
A split $n$-valued map $\phi\colon\thinspace X\multimap X$ may be deformed within $\splitmap{X}{X}{n}$   to a fixed point free $n$-valued map
if and only if any lift $\widehat{\Phi}=(f_1,\ldots,f_n)\colon\thinspace  X \to F_n(X)$ of $\phi$ may be deformed within $F_{n}(X)$
to a fixed point free map $\widehat{\Phi}'=(f_1',\ldots,f_n')\colon\thinspace  X \to F_n(X)$. In particular, for all $1\leq i\leq n$, there exists a homotopy $H_i\colon\thinspace X\times I\to X$ between $f_i$ and  $f_i'$, where $f_i'$ is a fixed point free map, and $H_j(x, t)\ne H_k(x, t)$ for all $1\leq j<k\leq n$, $x\in X$ and  $t\in [0,1]$.
\end{cor}

\begin{proof}
The `if' part of the statement may be obtained by considering the composition of the deformation between ${\Psi}$ and ${\Psi}'$ by the projection $\pi$ and by applying \reth{metriccont}.
The `only if' part follows in a manner similar to that of the proof of the first part of \relem{split}(\ref{it:splitIII}).
\end{proof}
 
\subsection{The fixed point property of simply connected spaces and the $k$-disc $\dt[k]$ for $n$-valued maps} \label{sec:disc} 

In this section, we analyse the case where $X$ is a simply-connected metric space that possesses the fixed point property, such as the closed $k$-dimensional disc $\dt[k]$. In \relem{split1}, we begin by proving a variant of the so-called `Splitting Lemma' that is more general than the versions that appear in the literature, such as that of Schirmer given in~\cite[Section~2, Lemma~1]{Sch0} for example. The hypotheses are expressed in terms of the homomorphism on the level of the fundamental group of the target $Y$, rather than that of the domain $X$, and the criterion is an algebraic condition,  in terms of the fundamental group, for an $n$-valued map from $X$ to $Y$ to be split. This allows us to prove \reth{sccfpp}, which says that a simply-connected metric space that has the fixed point property also possesses the fixed point property for $n$-valued maps for all $n\geq 1$. In particular, $\dt[k]$ satisfies this property for all $k\geq 1$. The $2$-disc will be the only surface with boundary that will be considered in this paper. The cases of other surfaces with boundary, such as the annulus and the M\"obius band, will be studied elsewhere.   

\begin{lem}\label{lem:split1}
Let $n\geq 1$, let $\phi\colon\thinspace  X \multimap Y$ be an $n$-valued map between metric spaces, where $X$ is connected and locally arcwise-connected, and let $\Phi\colon\thinspace X \to D_n(Y)$ be the associated $n$-unordered map. Then $\phi$ is split if and only if the image of the induced homomorphism $\Phi_{\#}\colon\thinspace \pi_1(X) \to B_n(Y)$ is contained in the image of the homomorphism $\pi_{\#}\colon\thinspace  P_n(Y) \to B_{n}(Y)$ induced by the covering map $\pi\colon\thinspace F_{n}(Y)\to D_{n}(Y)$. In particular, if $X$ is simply connected then all $n$-valued maps from $X$ to $Y$ are split.
 \end{lem}
 
\begin{proof}
Since $X$ and $Y$ are metric spaces, using \reth{metriccont},  we may consider the $n$-unordered map $\Phi\colon\thinspace X \to D_n(Y)$ that corresponds to $\phi$. The first part of the statement follows from standard results about the lifting property of a map to a covering space in terms of the fundamental group~\cite[Chapter~5, Section~5, Theorem~5.1]{Mas}. The second part is a consequence of the first part.
\end{proof}
 
As a consequence  of \relem{split1}, we are able to prove~\reth{sccfpp}. 
 
\begin{proof}[Proof of \reth{sccfpp}]
Let $X$ be a simply-connected metric space that has the fixed point property.  By \relem{split1}, any $n$-valued map $\phi\colon\thinspace X \multimap X$ is split. Writing $\phi=\{f_1,\ldots,f_n\}$, each of the maps $f_i\colon\thinspace X \to X$ is a self-map of $X$ that has at least one  fixed point. So $\phi$ has at least $n$ fixed points, and in particular, $X$ has the fixed point property for $n$-valued maps. The last part of the statement then follows.
\end{proof}

\subsection{$n$-valued maps of the sphere $\mathbb S^{2k}$}\label{sec:sph}

Let $k\geq 1$. Although $\mathbb S^{2k}$ does not have the fixed point property for self-maps, we shall show in this section that it has the fixed point property for $n$-valued maps  for all $n>1$, which is the statement of \repr{S2fp}. We first prove a lemma.

\begin{lem}\label{lem:S2split}
Let $n\geq 1$ and $k\geq 2$. Then any $n$-valued map of $\mathbb S^{k}$ is split.
\end{lem}

\begin{proof}
The result follows from \relem{split1} using the fact that $\mathbb S^k$ is simply connected.
\end{proof}

\begin{proof}[Proof of \repr{S2fp}]
Let $n\geq 2$, and let $\phi \colon\thinspace  \mathbb S^{2k} \multimap \mathbb S^{2k}$ be an $n$-valued map. By \relem{S2split}, $\phi$ is split, so it admits a lift $\widehat{\Phi}\colon\thinspace  \mathbb S^{2k} \to F_{n}(\St[2k])$, where $\widehat{\Phi}=(f_1,f_2,\ldots,f_n)$. Since $f_1(x)\ne f_2(x)$ for 
all $x\in \St[2k]$, we have $f_2(x)\neq -(-f_1(x))$, it follows that $f_2$ is homotopic to $-f_1$ via a homotopy that for all $x\in \mathbb S^{2k}$, takes $-f_1(x)$ to $f_2(x)$ along the unique geodesic that joins them. Thus the degree of one of the maps $f_1$ and $f_2$ is different from $-1$, and so has a fixed point, which implies that $\phi$ has a fixed point.
\end{proof}

\begin{rem}
If $n>2$ and $k=1$ then the result of \repr{S2fp} is clearly true since by~\cite[pp.~43--44]{GG5}, the set $[\St, F_{n}(\St)]$ of homotopy classes of maps between $\St$ and $F_{n}(\St)$ contains only one class, which is that of the constant map. So any representative of this class is of the form $\phi=(f_1,\ldots,f_n)$, where all of the maps $f_i\colon\thinspace \St\to \St$ are homotopic to the constant map. Such a map always has a  fixed point, and hence $\phi$ has at least $n$ fixed points. 
\end{rem}

\section{$n$-valued maps of the projective space $\R P^{2k}$}\label{sec:rp2k}

In this section, we will show that the projective space $\R P^{2k}$ also has the fixed point property for $n$-valued maps, which is the statement of \reth{rp2Kfpp}. Since $\R P^{2k}$ is not simply connected, we will require more elaborate arguments than those used in Sections~\ref{sec:disc} and~\ref{sec:sph}. We separate the discussion into two cases, $k=1$, and $k>1$.

\subsection{$n$-valued maps of \rp}\label{sec:rp2ka}

The following result is the analogue of \relem{S2split} for $\rp$. 

\begin{lem}\label{lem:rp2split}
 Let $n\geq 1$. Then any $n$-valued map of the projective plane is split.
\end{lem}
 
\begin{proof}
Let $\phi\colon\thinspace \rp \multimap \rp$ be an $n$-valued map, let $\Phi\colon\thinspace \rp \to D_n(\rp)$ be the associated $n$-unordered map, and let $\Phi_{\#}\colon\thinspace  \pi_{1}(\rp) \to B_{n}(\rp)$ be the homomorphism induced on the level of fundamental groups. Since $\pi_{1}(\rp)$ is isomorphic to the cyclic group of order $2$, it follows that $\im{\Phi_{\#}}$ is contained in the subgroup $\ang{\ft}$ of $B_{n}(\rp)$ generated by the full twist braid $\ft$ of $B_n(\rp)$ because $\ft$ is the unique element of $B_n(\rp)$ of order $2$ \cite[Proposition~23]{GG3}. But $\ft$ is a pure braid, so $\im{\Phi_{\#}}\subset P_{n}(\rp)$. Thus $\Phi$ factors through $\pi$, and hence $\phi$ is split by \relem{split1} as required.
\end{proof}

\begin{prop}\label{prop:RP2fp}
Let $n\geq 1$. Then any $n$-valued map of $\rp$ has at least $n$ fixed points, in particular $\rp$ has the fixed point property for $n$-valued maps.
\end{prop}

\begin{proof}
Let $\phi\colon\thinspace \rp \multimap \rp$ be an $n$-valued map of $\rp$. Then $\phi$ is split by \relem{rp2split}, and so $\phi=\brak{f_1,\ldots,f_n}$, where $f_{1},\ldots, f_{n}\colon\thinspace \rp \to \rp$ are pairwise coincidence-free self-maps of $\rp$. But $\rp$ has the fixed point property, and so for $i=1,\ldots,n$, $f_{i}$ has a fixed point. Hence $\phi$ has at least $n$ fixed points. 
\end{proof}

\subsection{$n$-valued maps of  $\R P^{2k}$,  $k>1$}\label{sec:proje}

The aim of this section is to prove that $\R P^{2k}$ has the fixed point property for $n$-valued maps for all  $n\geq 1$ and $k>1$. Indeed, we will show that every such $n$-valued map has at least $n$ fixed points. 

Given an $n$-valued map $\phi\colon\thinspace  X \multimap X$ of a topological space $X$, we consider the corresponding map $\Phi\colon\thinspace X \to D_n(X)$, and the induced homomorphism $\Phi_{\#}\colon\thinspace \pi_1(X) \to \pi_1(D_n(X))$ on the level of fundamental groups, where $\pi_1(D_n(X))=B_n(X)$. By the short exact sequence~\reqref{sesbraid}, $P_n(X)$ is a normal subgroup of $B_n(X)$ of finite index $n!$, so the subgroup $H=\Phi_{\#}^{-1}(P_n(X))$ is a normal subgroup of $\pi_1(X)$ of finite index. Further, if $L=\pi_1(X)/H$, the composition $\pi_1(X) \stackrel{\Phi_{\#}}{\to} B_n(X) \stackrel{\tau}{\to} S_n$ is a homomorphism that induces a homomorphism between $L$ and $S_n$. 

\begin{prop}\label{prop:nielsen}
Let $n\in \N$. Suppose that $X$ is a connected, locally  arcwise-connected metric space.
With the above notation, there exists a covering $q\colon\thinspace  \widehat{X} \to X$ of $X$ that corresponds to the subgroup $H$, and the $n$-valued map $\phi_1=\phi \circ q\colon\thinspace  \widehat{X}\multimap X$ admits exactly $n!$ lifts, which are  $n$-ordered  maps from $\widehat{X}$ to $F_n(X)$. If one such lift $\widehat{\Phi}_{1}\colon\thinspace \widehat{X}\to F_n(X)$ is given by $\widehat{\Phi}_{1}=(f_1,\ldots, f_n)$, where for $i=1,\ldots,n$, $f_i$ is a map from $\widehat{X}$ to $X$, then the other lifts are of the form $(f_{\tau(1)},\ldots,f_{\tau(n)})$, where $\tau\in S_n$.  
\end{prop} 

\begin{proof}
The first part is a consequence of~\cite[Theorem~5.1, Chapter~V, Section~5]{Mas}, using the observation that  $S_{n}$ is the deck transformation group corresponding to the covering $\pi\colon\thinspace F_n(X) \to D_n(X)$. The second part follows from the fact that $S_n$ acts freely on the covering space $F_n(X)$ by permuting coordinates.
\end{proof}

The fixed points of the $n$-valued map $\phi\colon\thinspace X\multimap X$ may be described in terms of the coincidences of the covering map $q\colon\thinspace \widehat{X} \to X$ with the maps $f_{1},\ldots, f_{n}$ given in the statement of \repr{nielsen}. 

\begin{prop}\label{prop:coinfix} Let $n\in \N$, let
 $X$ be a connected, locally arcwise-connected, metric space, let $\phi\colon\thinspace  X \multimap X$ be an $n$-valued  map, and let ${\widehat \Phi_1}=(f_1,\ldots,f_n)\colon\thinspace  \widehat{X} \to F_n(X)$ be an $n$-ordered  map that is a lift of $\phi_1=\phi\circ q\colon\thinspace \widehat{X} \multimap X$ as in \repr{nielsen}. Then the map $q$ restricts to a surjection $q\colon\thinspace \bigcup_{i=1}^{n} \operatorname{\text{Coin}}(q, f_i) \to \operatorname{\text{Fix}}(\phi)$.
Furthermore, the pre-image of a point $x\in \operatorname{\text{Fix}}(\phi)$ by this map is precisely $q^{-1}(x)$, namely the fibre over $x\in X$ of the covering map $q\colon\thinspace \widehat{X} \to X$. 
\end{prop}

\begin{proof}
Let $\widehat{x}\in \operatorname{\text{Coin}}(q, f_i)$ for some $1\leq i\leq n$, and let $x=q(\widehat{x})$. Then $f_i(\widehat{x})=q(\widehat{x})$, and since $\Phi(x)= \pi\circ \widehat{\Phi}_1(\widehat{x})=\brak{f_1(\widehat{x}),\ldots,f_n(\widehat{x})}$, it follows that $x\in \phi(x)$, \emph{i.e.}\ $x\in \operatorname{\text{Fix}}(\phi)$, so the map is well defined. To prove surjectivity and the second part of the statement, it suffices to show that if $x \in \operatorname{\text{Fix}}(\phi)$, then any element $\widehat{x}$ of $q^{-1}(x)$ belongs to $\bigcup_{i=1}^{n} \operatorname{\text{Coin}}(q, f_i) $. So let $x\in \phi(x)$, and let $\widehat{x}\in \widehat{X}$  be such that $q(\widehat{x})=x$. By commutativity of the following diagram:
\begin{equation*}
\begin{tikzcd}[ampersand replacement=\&]
\&\& F_{n}(X) \ar{d}{\pi}\\
\widehat{X} \ar[swap]{r}{q} \ar[dashrightarrow, end anchor=south west]{rru}{\widehat{\Phi}_1}  \& X \ar[swap]{r}{\Phi} \& D_{n}(X),
\end{tikzcd}
\end{equation*}
and the fact that $x\in \operatorname{\text{Fix}}(\phi)$, it follows that $x$ is one of the coordinates, the $j\up{th}$ coordinate say, of $\widehat{\Phi}_1(\widehat{x})$. This implies that $\widehat{x}\in \operatorname{\text{Coin}}(q, f_j)$, which completes the proof of the proposition.
\end{proof}

\begin{prop}\label{prop:nullhomo}
Let $n,k>1$. If $\phi\colon\thinspace  \St[2k] \multimap \R P^{2k}$ is an $n$-valued map, then $\phi$ is split, and for $i=1,\ldots,n$, there exist maps $f_i\colon\thinspace  \St[2k] \to \R P^{2k}$ for which  $\phi=\{f_1,\ldots,f_n\}$. Further, $f_i$ is null homotopic for all $i\in \brak{1,\ldots,n}$. 
\end{prop}

\begin{proof}
The first part follows from \relem{split1}. It remains to prove the second part, \emph{i.e.}\ that each $f_i$ is null homotopic. Since $\St[2k]$ is simply connected, the set $[\mathbb{S}^{2k}, \mathbb{S}^{2k}]=[\mathbb{S}^{2k}, \mathbb{S}^{2k}]_{0}=\pi_{2k}(\mathbb{S}^{2k})$, where $[ \cdot , \cdot  ]_{0}$ denotes basepoint-preserving homotopy classes of maps.  Let $x_0\in  \mathbb{S}^{2k}$ be a basepoint, let $p\colon\thinspace \St[2k] \to \R P^{2k}$ be the two-fold covering, and let $\overline{x}_0=p(x_{0})\in \R P^{2k}$ be the basepoint of $\R P^{2k}$. Consider the  natural map $[(\mathbb{S}^{2k}, x_0), (\mathbb{S}^{2k}, x_0)] \to [(\mathbb{S}^{2k}, x_0), (\R P^{2k}, \overline{x}_0)]$ from the set of  based  homotopy classes of self-maps of $\St[2k]$ to the based homotopy classes of maps from $\St[2k]$ to $\R P^{2k}$,
that to a homotopy class of a basepoint-preserving self-map of $\St[2k]$ associates the homotopy class of the composition of this self-map with $p$.
This correspondence is an isomorphism. The covering map $p$ has topological degree $2$~\cite{Ep}, so the degree of a map $f\colon\thinspace \St[2k] \to \R P^{2k}$ is an even integer (we use the system of local coefficients given by the orientation of $\R P^{2k}$~\cite{Ol}). Since $H^{l}( \R P^{2k} ,  \widetilde{\mathbb{Q}})=0$ for $\widetilde{\mathbb{Q}}$ twisted by the orientation and  $l\ne 2k$, if $i\neq j$, it follows that the Lefschetz coincidence number $L(f_{i},f_{j})$ is equal to $\deg(f_i)$. But $f_i$ and $f_j$ are coincidence free, so their Lefschetz coincidence number must be zero,  
which implies that $\deg(f_i)=0$~\cite{GJ}. Since $n>1$, we conclude that $\deg(f_i)=0$ for all $1\leq i\leq n$, and the result follows.
\end{proof}
 
We are now able to prove the main result of this section, that $\R P^{2k}$ has the fixed point property for $n$-valued maps for all $k,n\geq 1$.

\begin{proof}[Proof of \reth{rp2Kfpp}]
The case $n=1$ is classical, so assume that $n\geq 2$. We use the notation introduced at the beginning of this section, taking $X=\R P^{2k}$. Since $\pi_1(\R P^{2k})\cong\Z_2$, $H$ is either $\pi_1(\R P^{2k})$ or the trivial group.  In the former case, the $n$-valued map $\phi\colon\thinspace \R P^{2k} \multimap \R P^{2k}$ is split, $\operatorname{\text{Fix}}(\phi)=\bigcup_{i=1}^{n} \operatorname{\text{Fix}}(f_i)$, 
where $\phi=\brak{f_1,\ldots,f_n}$, and for all $i=1,\ldots,n$, $f_{i}$ is a self-map of $\R P^{2k}$. It follows that $\phi$ has at least $n$ fixed points. So suppose that $H$ is the trivial subgroup of $\pi_1(\R P^{2k})$. Then $\widehat{\R P^{2k}}=\St[2k]$, and $q$ is the covering map $p\colon\thinspace \St[2k] \to \R P^{2k}$.
We first consider the case $n=2$. Let $\phi\colon\thinspace \R P^{2k} \multimap \R P^{2k}$ be a $2$-valued map, and let $\widehat{\Phi}_1\colon\thinspace \mathbb{S}^{2k} \to F_2(\R P^{2k})$ be a lift of the map $\Phi_1=\Phi\circ p\colon\thinspace \mathbb{S}^{2k} \to D_2(\R P^{2k})$ that factors through the projection $\pi\colon\thinspace F_2(\R P^{2k}) \to D_2(\R P^{2k})$. By \repr{nielsen}, $\widehat{\Phi}_1=(f_1, f_2)$, where for $i=1,2$, $f_i\colon\thinspace \mathbb{S}^{2k} \to \R P^{2k}$ is a single-valued map, and it follows from \repr{nullhomo} that $f_1$ and $f_2$ are null homotopic. 
If $\operatorname{\text{Coin}}(f_i, p)= \vide$ for some $i\in \brak{1,2}$, then arguing as in the second part of the proof of \repr{nullhomo}, it follows that $L(p,f_i)=\deg(p)=2$, which yields a contradiction. So $\operatorname{\text{Coin}}(f_i, p)\ne \vide$ for all $i\in \brak{1,2}$. Using the fact that $f_{1}$ and $f_{2}$ are coincidence free, we conclude that $\phi$ has at least two fixed points, and the result follows in this case.

Finally suppose that $n>2$. Arguing as in the case $n=2$, we obtain a lift of $\phi\circ p$ of the form $(f_1, \ldots, f_n)$, where for $i=1,\ldots,n$,  
$f_i\colon\thinspace \mathbb{S}^{2k} \to \R P^{2k}$ is a map. If $i=1,\ldots,n$ then for all $j\in \brak{1,\ldots,n}$, $j\neq i$, we may apply the above argument to $f_i$ and $f_j$ to obtain $\operatorname{\text{Coin}}(f_i, p)\ne \vide$. Hence $\phi$ has at least $n$ fixed points, and the result follows.
\end{proof}

\begin{rem} 
If $n>1$, we do not know whether there exists a non-split $n$-valued map $\phi\colon\thinspace  \R P^{2k} \multimap \R P^{2k}$. 
\end{rem}

\section{Deforming (split) $n$-valued maps to fixed point and root-free maps}\label{sec:fixfree}

In this section, we generalise a standard procedure for deciding whether a single-valued map may be deformed to a fixed point free map to the $n$-valued case. We start by giving a necessary and sufficient condition for an $n$-valued map (resp.\ a split $n$-valued map) $\phi\colon\thinspace X \to X$ to be deformable to a fixed point free $n$-valued map (resp.\ split fixed point free $n$-valued map), at least in the case where $X$ is a manifold without boundary. This enables us to prove \reth{defchineg}. We then go on to give the analogous statements for roots. Recall from \resec{intro} that $D_{1,n}(X)$ is the quotient of $F_{n+1}(X)$ by the action of the subgroup $\brak{1} \times S_n$ of the symmetric group $S_{n+1}$, and that $B_{1,n}(X)=\pi_1(D_{1,n}(X))$.  

\begin{prop}\label{prop:defor} Let $n\in \N$, let $X$ be a metric space, and let $\phi\colon\thinspace  X \multimap X$ be an $n$-valued map. 
If $\phi$ can be deformed to a fixed point free $n$-valued map, then there exists a map $\Theta\colon\thinspace  X\to D_{1,n}(X)$ such that the following diagram is commutative up to homotopy: 
\begin{equation}\label{eq:commdiag3}
\begin{tikzcd}[ampersand replacement=\&]
\& \& D_{1,n}(X) \ar{d}{\iota_{1,n}}\\
X \ar[swap]{rr}{\id_{X}\times \Phi} \ar[dashrightarrow, end anchor=south west]{rru}{\Theta}
\& \& X \times D_{n}(X),
\end{tikzcd}
\end{equation}
where $\iota_{1,n}\colon\thinspace D_{1,n}(X) \to X \times D_n(X)$ is the inclusion map. Conversely, if $X$ is a manifold without boundary and there exists a map $\Theta\colon\thinspace  X\to D_{1,n}(X)$ such that diagram~\reqref{commdiag3} is commutative up to homotopy, then the $n$-valued map $\phi\colon\thinspace  X \multimap X$ may be deformed to a fixed point free $n$-valued map.
\end{prop}

\begin{proof}
For the first part, if $\phi'$ is a fixed point free deformation of $\phi$, then we may take the factorisation map $\Theta\colon\thinspace  X\to D_{1,n}(X)$ to be that defined by $\Theta(x)= (x, \Phi'(x))$. For the converse, the argument  is similar to the proof of the case of single-valued maps, and is as follows. 
Let $\Theta\colon\thinspace X \to D_{1,n}(X)$ be a homotopy factorisation that satisfies the hypotheses. Composing $\Theta$ with the projection $\overline{p}_{1,n}\colon\thinspace  D_{1,n}(X)\to X$ onto the first coordinate, we obtain a self-map of $X$ that is homotopic to the identity. Let $H\colon\thinspace  X\times I \to X$ be a homotopy between $\overline{p}_{1,n}\circ \Theta$   and  $\id_{X}$. Since $\overline{p}_{1,n}\colon\thinspace  D_{1,n}(X)\to X$ is a fibration and there is a lift of the restriction of $H$ to $X\times \brak{0}$, $H$ lifts to a homotopy $\widetilde{H}\colon\thinspace  X\times I \to  D_{1,n}(X)$. The restriction of $\widetilde{H}$ to $X\times \brak{1}$ yields the required deformation.
\end{proof}

For split $n$-valued maps, the correspondence given by \relem{split}(\ref{it:splitII}) gives rise to a statement analogous to that of \repr{defor} in terms of $F_{n}(X)$.
 
\begin{prop}\label{prop:equivsplit} Let $n\in \N$, let $X$ be a metric space, and let $\phi\colon\thinspace  X \multimap X$ be a split $n$-valued map. If $\phi$ can be deformed  to a fixed point free $n$-valued map, and if $\widehat{\Phi}\colon\thinspace X \to F_{n}(X)$ is a lift of $\phi$, then there exists a map $\widehat{\Theta}\colon\thinspace  X \to F_{n+1}(X)$ such that the following diagram is commutative up to homotopy: 
\begin{equation}\label{eq:commdiag4}
\begin{tikzcd}[ampersand replacement=\&]
\&\& F_{n+1}(X) \ar{d}{\widehat{\iota}_{n+1}}\\
X \ar[swap]{rr}{\id_{X}\times \widehat{\Phi}}  \ar[dashrightarrow, end anchor=south west]{rru}{\widehat{\Theta}}  \&\& X \times F_{n}(X),
\end{tikzcd}
\end{equation}
where $\widehat{\iota}_{n+1}\colon\thinspace F_{n+1}(X) \to X \times F_n(X)$ is the inclusion map. Conversely, if $X$ is a manifold without boundary and there exists a map $\widehat{\Theta}\colon\thinspace  X\to F_{n+1}(X)$ such that diagram~\reqref{commdiag4} is commutative up to homotopy, then the split $n$-valued map  $\phi\colon\thinspace  X \multimap X$ may be deformed through split maps to a  fixed point free split $n$-valued map.
\end{prop}

\begin{proof}
Similar to that of \repr{defor}.
\end{proof}
 
We now apply Propositions~\ref{prop:defor} and~\ref{prop:equivsplit} to prove \reth{defchineg}, which treats the case where $X$ is a compact  surface without boundary (orientable or not) of non-positive Euler characteristic. 
 
\begin{proof}[Proof of \reth{defchineg}]  The space  $D_{1,n}(X)$ is a finite covering of $D_{1+n}(X)$, so it is a   $K(\pi, 1)$ since 
 $D_{1+n}(X)$ is a $K(\pi, 1)$ by~\cite[Corollary~2.2]{FaN}.  To prove the `only if' implication of part~(\ref{it:defchinega}),  diagram~\reqref{commdiag1} 
implies the existence of diagram~\reqref{commdiag3} using the fact that the space $D_{1,n}(X)$ is a $K(\pi, 1)$, 
where $\varphi=\Theta_{\#}$ is the homomorphism induced by $\Theta$  on the level of the fundamental groups. Conversely, diagram~\reqref{commdiag3} implies that the two maps $\iota_{1,n}\circ \Theta$ and $\id_X\times \Phi$ are homotopic, but not necessarily by a basepoint-preserving homotopy, so diagram~\reqref{commdiag1} is commutative up to conjugacy. Let $\delta\in\pi_{1}(X) \times B_{n}(X)$ be such that $(\iota_{1,n})_{\#}\circ \varphi(\alpha)=\delta (\id_X\times \phi)_{\#}(\alpha) \delta^{-1}$ for all $\alpha\in \pi_1(X)$, and let $\widehat{\delta}\in B_{1,n}(X)$ be an element  such that  $(\iota_{1,n})_{\#}(\widehat{\delta})=\delta.$ Considering the homomorphism  $\varphi'\colon\thinspace \pi_1(X) \to B_{1,n}(X)$ defined by $\varphi'(\alpha)=\widehat{\delta}^{-1}\varphi (\alpha)\widehat{\delta}$ for all $\alpha\in \pi_1(X)$, we obtain the commutative diagram~\reqref{commdiag1}, where we replace $\varphi$ by $\varphi'$. The proof of part~(\ref{it:defchinegb}) is similar, and is left to the reader. 
\end{proof}

For the case of roots, we now give statements analogous to those of Propositions~\ref{prop:defor} and~\ref{prop:equivsplit} and of \reth{defchineg}. The proofs are similar to those of the corresponding statements for fixed points, and the details are left to the reader.  
 
\begin{prop}\label{prop:rootI} Let $n\in \N$, let $X$ and $Y$ be metric spaces, let $y_0\in Y$ be a basepoint, and let $\phi\colon\thinspace  X \multimap Y$ be an $n$-valued map. If $\phi$ can be deformed to a root-free $n$-valued map, then there exists a map $\Theta\colon\thinspace  X\to D_{n}(Y\backslash\{y_0\})$ such that the following diagram is commutative up to homotopy: 
\begin{equation}\label{eq:commdiag3I}
\begin{tikzcd}[ampersand replacement=\&]
\& \&  D_{n}(Y\backslash\{y_0\}) \ar{d}{\iota_{n}}\\
X \ar[swap]{rr}{ \Phi} \ar[dashrightarrow, end anchor=south west]{rru}{\Theta}
\& \&  D_{n}(Y),
\end{tikzcd}
\end{equation}
where the map $\iota_{n}\colon\thinspace D_{n}(Y\backslash\{y_0\}) \to   D_n(Y)$ is induced by the inclusion map $Y\backslash \{y_0\} \lhra Y$. Conversely, if $Y$ is a manifold without boundary and there exists a map $\Theta\colon\thinspace  X\to D_{n}(Y)$ such that diagram~\reqref{commdiag3I} is commutative up to homotopy, then the $n$-valued map $\phi\colon\thinspace  X \multimap Y$ may be deformed to a root-free $n$-valued map.
\end{prop}

For split $n$-valued maps, the correspondence of  \relem{split}(\ref{it:splitII}) gives rise to a statement analogous to that of Proposition \ref{prop:rootI} in terms of $F_{n}(Y)$.
 
\begin{prop}\label{prop:equivsplitI}  Let $n\in \N$, let $X$ and $Y$ be   metric  spaces, let $y_0\in Y$ a basepoint, and let $\phi\colon\thinspace  X \multimap Y$ be a split $n$-valued map. If $\phi$ can be deformed  to a root-free $n$-valued map then there exists a map $\widehat{\Theta}\colon\thinspace  X \to F_{n}(Y\backslash \{y_0\})$  and a lift $\widehat{\Phi}$ of $\phi$ such that the following diagram is commutative up to homotopy: 
\begin{equation}\label{eq:commdiag4I}
\begin{tikzcd}[ampersand replacement=\&]
\&\& F_{n}(Y\backslash \{y_0\}) \ar{d}{\widehat{\iota}_{n}}\\
X \ar[swap]{rr}{ \widehat{\Phi}}  \ar[dashrightarrow, 
end anchor=south west]{rru}{\widehat{\Theta}}  \&\&  F_{n}(Y),
\end{tikzcd}
\end{equation}
where the map $\widehat{\iota}_{n}\colon\thinspace F_{n}(Y\backslash \{y_0\}) \to  F_n(Y)$ is induced by the inclusion map $Y\backslash \{y_0\} \lhra Y$. Conversely, if $Y$ is a manifold without boundary, and there exists a map $\widehat{\Theta}\colon\thinspace  X\to F_{n}(Y\backslash \{y_0\})$ such that diagram~\reqref{commdiag4I} is commutative up to homotopy, then the split $n$-valued map  $\phi\colon\thinspace  X \multimap Y$ may be deformed through split maps to a root-free split $n$-valued map.
\end{prop}

Propositions~\ref{prop:rootI} and~\ref{prop:equivsplitI}  may be applied to the case where $X$ and $Y$ are  compact surfaces without boundary of non-positive Euler characteristic to obtain the analogue of \reth{defchineg} for roots.

\begin{thm}\label{th:defchinegI} 
Let $n\in \N$, and let $X$ and $Y$ be compact surfaces without boundary of non-positive Euler characteristic.
\begin{enumerate}[(a)]
\item\label{it:defchinegal} An $n$-valued map 
$\phi\colon\thinspace  X \multimap Y$ can be deformed to a root-free $n$-valued  map if and only if there is a homomorphism 
$\varphi\colon\thinspace  \pi_1(X) \to B_{n}(Y\backslash\{y_0\})$ that makes the following diagram commute: 

\begin{equation*}
\begin{tikzcd}[ampersand replacement=\&]
\&\& B_{n}(Y\backslash\{y_0\}) \ar{d}{(\iota_{n})_{\#}}\\
\pi_{1}(X) \ar[swap]{rr}{\Phi_{\#}}  \ar[dashrightarrow, end anchor=south west]{rru}{\varphi}  \&\& B_{n}(Y),
\end{tikzcd}
\end{equation*}
where $\iota_{n}\colon\thinspace D_{n}(Y\backslash\{y_0\}) \to  D_{n}(Y)$ is induced by the inclusion map $Y\backslash \{y_0\} \lhra Y$, and $\Phi\colon\thinspace X\to D_{n}(Y)$ is the $n$-unordered map associated to $\phi$.

\item\label{it:defchinegbI} A split $n$-valued map $\phi\colon\thinspace  X \multimap Y$ can be deformed to a root-free $n$-valued map if and only if there exist a lift $\widehat{\Phi}\colon\thinspace X \to F_n(Y)$ of $\phi$ and a homomorphism $\widehat{\varphi}\colon\thinspace  \pi_1(X) \to P_{n}(Y\backslash\{y_0\})$ that make the following diagram commute:
\begin{equation*}
\begin{tikzcd}[ampersand replacement=\&]
\&\& P_{n}(Y\backslash\{y_0\}) \ar{d}{(\widehat{\iota}_{n})_{\#}}\\
\pi_{1}(X) \ar[swap]{rr}{\widehat{\Phi}_{\#}}  \ar[dashrightarrow, end anchor=south west]{rru}{\widehat{\varphi}}  \&\&  P_{n}(Y),
\end{tikzcd}
\end{equation*} 
where $\widehat{\iota}_{n}\colon\thinspace F_{n}(Y\backslash\{y_0\}) \to  F_{n}(Y)$ is induced by the inclusion map $Y\backslash \{y_0\} \lhra Y$.
\end{enumerate}
\end{thm}

\section{An application to split $2$-valued maps of the $2$-torus}\label{sec:toro}

In this section, we will use some of the ideas and results of \resec{fixfree} to study the fixed point theory of $2$-valued maps of the $2$-torus $\T$. We restrict our attention to the case where the maps are split, \emph{i.e.}\ we consider $2$-valued maps of the form $\phi\colon\thinspace  \T \multimap \T$ that admit a lift $\widehat{\Phi}\colon\thinspace  \T \to F_{2}(\T)$,  where $\widehat{\Phi}=(f_1, f_2)$, $f_1$ and $f_2$ being coincidence-free self-maps of $\T$. We classify the set of homotopy classes of split $2$-valued maps of $\T$, and we study the question of the characterisation of those split $2$-valued maps that can be deformed to fixed point free $2$-valued maps. The case of arbitrary $2$-valued maps of $\T$ will be treated in a forthcoming paper. In \resec{toro2}, we give presentations of the groups $P_{2}(\T)$, $B_{2}(\T)$ and $P_{2}(\T\setminus\brak{1})$ that will be used in the following sections, where $1$ denotes a basepoint of $\T$. In \resec{descript}, we describe the set of based and free homotopy classes of split $2$-valued maps of $\T$.  In \resec{fptsplit2}, we give a formula for the Nielsen number, and we derive a necessary condition for such a split $2$-valued map to be deformable to a fixed point free $2$-valued map. We then give an infinite family of homotopy classes of 
split $2$-valued maps of $\T$ that satisfy this condition and that 
may be deformed to fixed point free $2$-valued maps. To facilitate the calculations, in \resec{p2Tminus1}, we shall show that the fixed point problem is equivalent to a root problem.

\subsection{The groups  $P_{2}(\T)$, $B_{2}(\T)$ and $P_{2}(\T\setminus\brak{1})$}\label{sec:toro2}

In this section, we give presentations of $P_{2}(\T)$, $B_{2}(\T)$ and $P_{2}(\T\setminus\brak{1})$ that will be used in the following sections. Other presentations of these groups may be found in the literature, see~\cite{Bel,Bi,GM,Sco} for example. We start by considering the group $P_{2}(\T\setminus\brak{1})$. If $u$ and $v$ are elements of a group $G$, we denote their commutator $uvu^{-1}v^{-1}$ by $[u,v]$, and the commutator subgroup of $G$ by $\Gamma_{2}(G)$. If $A$ is a subset of $G$ then $\ang{\!\ang{A}\!}_{G}$ will denote the normal closure of $A$ in $G$. 

\begin{prop}\label{prop:presP2T1}
The group $P_{2}(\T\setminus\brak{1})$ admits the following presentation:
\begin{enumerate}
\item[generators:] $\rho_{1,1}$, $\rho_{1,2}$, $\rho_{2,1}$,  $\rho_{2,2}$, $B_{1,2}$, $B$ and $B'$.
\item[relations:]\mbox{}
\begin{enumerate}[(a)]
\item\label{it:presP2T1a} $\rho_{2,1}\rho_{1,1}\rho_{2,1}^{-1}=B_{1,2}\rho_{1,1}B_{1,2}^{-1}$.
\item $\rho_{2,1}\rho_{1,2}\rho_{2,1}^{-1}=B_{1,2}\rho_{1,2}\rho_{1,1}^{-1}B_{1,2}\rho_{1,1}B_{1,2}^{-1}$.
\item $\rho_{2,2}\rho_{1,1}\rho_{2,2}^{-1}=\rho_{1,1}B_{1,2}^{-1}$.
\item $\rho_{2,2}\rho_{1,2}\rho_{2,2}^{-1}=B_{1,2}\rho_{1,2}B_{1,2}^{-1}$.
\item\label{it:presP2T1e} $\rho_{2,1}B\rho_{2,1}^{-1}=B$ and $\rho_{2,2}B\rho_{2,2}^{-1}=B$.
\item\label{it:presP2T1f} $\rho_{2,1}B_{1,2}\rho_{2,1}^{-1}=B_{1,2} \rho_{1,1}^{-1}B_{1,2} \rho_{1,1}B_{1,2}^{-1}$ and $\rho_{2,2}B_{1,2}\rho_{2,2}^{-1}=B_{1,2} \rho_{1,2}^{-1}B_{1,2} \rho_{1,2}B_{1,2}^{-1}$.
\item\label{it:relf} $B'\rho_{1,1}B'^{-1}=\rho_{1,1}$ and $B'\rho_{1,2}B'^{-1}=\rho_{1,2}$.
\item\label{it:relg} $B'B_{1,2}B'^{-1}=B_{1,2}^{-1}B^{-1} B_{1,2}BB_{1,2}$ and $B'BB'^{-1}=B_{1,2}^{-1}BB_{1,2}$.
\item\label{it:relh} $[\rho_{1,1},\rho_{1,2}^{-1}]=BB_{1,2}$ and $[\rho_{2,1},\rho_{2,2}^{-1}]=B_{1,2}B'$.
\end{enumerate}
\end{enumerate}
In particular, $P_{2}(\T\setminus\brak{1})$ is a semi-direct product of the free group of rank three generated by $\brak{\rho_{1,1},\rho_{1,2},B_{1,2}}$ by the free group of rank two generated by $\brak{\rho_{2,1},\rho_{2,2}}$.
\end{prop}

Geometric representatives of the generators of $P_{2}(\T\setminus\brak{1})$ are illustrated in Figure~\ref{fig:gens}. The torus is obtained from this figure by identifying the boundary to a point.

\begin{figure}[!h]
\hspace*{\fill}
\begin{tikzpicture}[>=stealth,scale=0.8]
\draw[line width=1.3mm] (0,0) circle(5);

\draw[fill,white] (4.24,2.67) circle(0.77);
\draw[fill,white] (-4.24,2.67) circle(0.77);

\draw[fill,white] (1.5,4.7) circle(0.77);
\draw[fill,white] (-1.5,4.7) circle(0.77);

\centerarc[line width=1.3mm,rotate around={-20:(0,0)}](0,5)(-23.5:203.5:4)
\centerarc[line width=1.3mm,rotate around={-20:(0,0)}](0,5)(-14.5:194.5:2.5)

\draw[fill,white] (0,6.5) circle(0.25);
\draw[fill,white] (0,8.3) circle(0.25);
\draw[fill,white] (-1.4,7.1) circle(0.25);
\draw[fill,white] (1.4,7.1) circle(0.25);

\draw[fill,white] (0.8,7) circle(0.8);
\draw[fill,white] (-0.7,7.6) circle(0.8);

\draw[fill] (-1.5,0) circle(0.12);
\draw[fill] (1.5,0) circle(0.12);

\draw[fill] (1.5,-3) circle(0.1);
\draw[fill,white] (1.5,-3) circle(0.04);

\begin{scope}[shift={(-1.5,0)}, rotate=0]
\draw[middlearrowrev=0.7] (0,0) .. controls (-5.5,3.5) and (-3.5,7.4) .. (-1.5,8.2);
\draw (0,0) .. controls (5.5,3.8) and (3,9) .. (-1.5,8.2);
\end{scope}

\begin{scope}[shift={(-1.5,0)}, rotate=-39]
\draw[middlearrowrev=0.4] (0,0) .. controls (-4.5,3.5) and (-4.5,8.4) .. (-1.5,9.2);
\draw (0,0) .. controls (4,3.8) and (3,10) .. (-1.5,9.2);
\end{scope}

\begin{scope}[shift={(-1.5,0)},rotate=-9]
\draw[middlearrowrev=0.4] (0,0) .. controls (1.5,0) and (3.8,-2) .. (3.9,0);
\draw(0,0) .. controls (1.5,0) and (3.8,2) .. (3.9,0);
\end{scope}

\begin{scope}[shift={(1.5,0)}, rotate=-4, ultra thick]
\draw (0,0) .. controls (4.5,4) and (3,6) .. (1.5,7.2);
\draw[middlearrowrev=0.3] (0,0) .. controls (-5.5,3.8) and (-2,9.5) .. (1.5,7.2);
\end{scope}

\begin{scope}[shift={(1.5,0)}, rotate=40, ultra thick]
\draw (0,0) .. controls (4.5,3.5) and (4.5,8) .. (1.5,8.5);
\draw[middlearrowrev=0.45] (0,0) .. controls (-4,3.8) and (-3,9) .. (1.5,8.5);
\end{scope}

\begin{scope}[shift={(-1.5,0)},rotate=-45]
\draw[middlearrowrev=0.4] (0,0) .. controls (1.5,0) and (4.3,-2) .. (4.8,0);
\draw(0,0) .. controls (1.5,0) and (4.3,1.8) .. (4.8,0);
\end{scope}

\begin{scope}[shift={(1.5,0)},rotate=-90, ultra thick]
\draw[middlearrowrev=0.5] (0,0) .. controls (1.5,0) and (3.8,-2) .. (3.9,0);
\draw(0,0) .. controls (1.5,0) and (3.8,2) .. (3.9,0);
\end{scope}

\draw[fill,white] (1.02,7.63) circle(0.2);
\draw[fill,white] (0.63,7.95) circle(0.2);

\draw[fill,white] (-0.25,6.75) circle(0.2);
\draw[fill,white] (-0.55,6.9) circle(0.2);

\centerarc[line width=1.3mm,rotate around={20:(0,0)}](0,5)(-23.5:203.5:4)
\centerarc[line width=1.3mm,rotate around={20:(0,0)}](0,5)(-14.5:194.5:2.5)

\node at (-3.5,0.95){$\rho_{1,1}$};
\node at (-2.25,3.3){$\rho_{1,2}$};
\node at (-3.1,2.4){$\rho_{2,1}$};
\node at (-0.5,4){$\rho_{2,2}$};
\node at (-0.75,-2){$B$};
\node at (2.75,-3){$B'$};
\node at (2.9,-0.3){$B_{1,2}$};
\node at (-1.8,-0.2){$x_{1}$};
\node at (1.1,-0.3){$x_{2}$};
\node at (1.2,-3){$1$};

\draw[fill] (-4.63,1.93) circle(0.08);
\draw[fill] (4.63,1.93) circle(0.08);
\draw[fill] (3.78,3.3) circle(0.08);
\draw[fill] (2.28,4.47) circle(0.08);
\draw[fill] (0.77,4.93) circle(0.08);
\draw[fill] (-3.78,3.3) circle(0.08);
\draw[fill] (-2.28,4.47) circle(0.08);
\draw[fill] (-0.77,4.93) circle(0.08);

\end{tikzpicture}
\hspace*{\fill}
\caption{The generators of  $P_{2}(\T\setminus\brak{1})$ given in \repr{presP2T1}.}
\label{fig:gens}
\end{figure}

\begin{rem}
The inclusion of $P_{2}(\T\setminus\{1\})$ in $P_{2}(\T)$ induces a surjective homomorphism $\map{\alpha}{P_2(\T\setminus\{1\})}{P_{2}(\T)}$ that sends $B$ and $B'$ to the trivial element, and sends each of the remaining generators (considered as an element of $P_2(\T\setminus\{1\})$) to itself (considered as an element of $P_2(\T)$). Applying this to the presentation of $P_{2}(\T\setminus\{1\})$ given by \repr{presP2T1}, we obtain the presentation of $P_{2}(\T)$ given in~\cite{FH}.
\end{rem}

\begin{proof}[Proof of \repr{presP2T1}]
Consider the following Fadell-Neuwirth short exact sequence:
\begin{equation}\label{eq:fnses}
1\to  P_{1}(\T\setminus\brak{1,x_{2}},x_{1}) \to P_{2}(\T\setminus\brak{1}, (x_{1},x_{2}))  \xrightarrow{(p_{2})_{\#}}  P_{1}(\T\setminus\brak{1}, x_{2}) \to 1,
\end{equation}
where $(p_{2})_{\#}$ is the homomorphism given geometrically by forgetting the first string and induced by the projection $p_{2}\colon\thinspace  F_{2}(\T\setminus\brak{1})\to \T\setminus\brak{1}$ onto the second coordinate. The kernel $K=P_{1}(\T\setminus\brak{1,x_{2}},x_{1})$ of $(p_{2})_{\#}$ (resp.\ the quotient $Q=P_{1}(\T\setminus\brak{1}, x_{1})$) is a free group of rank three (resp.\ two). It will be convenient to choose presentations for these two groups that have an extra generator. From Figure~\ref{fig:gens}, we take $K$ (resp.\ $Q$) to be generated by $X=\brak{\rho_{1,1},\rho_{1,2},B_{1,2},B}$ (resp.\ $Y=\brak{\rho_{2,1},\rho_{2,2},B'}$) subject to the single relation $[\rho_{1,1},\rho_{1,2}^{-1}]=BB_{1,2}$ (resp.\ $[\rho_{2,1},\rho_{2,2}^{-1}]=B'$). We apply standard methods to obtain a presentation of the group extension $P_{2}(\T\setminus\brak{1}, (x_{1},x_{2}))$~\cite[Proposition~1, p.~139]{Jo}. This group is generated by the union of $X$ with coset representatives of $Y$, which we take to be the same elements geometrically, but considered as elements of $P_{2}(\T\setminus\brak{1}, (x_{1},x_{2}))$. This yields the given generating set. There are three types of relation. The first is that of $K$. The second type of relation is obtained by lifting the relation of $Q$ to $P_{2}(\T\setminus\brak{1}, (x_{1},x_{2}))$, which gives rise to the relation $[\rho_{2,1},\rho_{2,2}^{-1}]B'^{-1}=B_{1,2}$. The third type of relation is obtained by rewriting the conjugates of the elements of $X$ by the chosen coset representatives of the elements of $Y$ in terms of the elements of $X$ using the geometric representatives of $X$ and $Y$ illustrated in Figure~\ref{fig:gens}. We leave the details to the reader. The last part of the statement is a consequence of the fact that $K$ (resp.\ $Q$) is a free group of rank three (resp.\ two), so the short exact sequence~\reqref{fnses} splits.
\end{proof}

\begin{rem}
For future purposes, it will be convenient to have the following relations at our disposal:
\begin{align*}
\rho_{2,1}^{-1}\rho_{1,1}\rho_{2,1}&=\rho_{1,1}B_{1,2}^{-1}\rho_{1,1}B_{1,2}\rho_{1,1}^{-1} &
\rho_{2,1}^{-1}\rho_{1,2}\rho_{2,1}&=\rho_{1,1}B_{1,2}^{-1}\rho_{1,1}^{-1} \rho_{1,2}B_{1,2}^{-1} \rho_{1,1}B_{1,2}\rho_{1,1}^{-1}\\
\rho_{2,2}^{-1}\rho_{1,1}\rho_{2,2}&=\rho_{1,1}\rho_{1,2}B_{1,2}\rho_{1,2}^{-1} &
\rho_{2,2}^{-1}\rho_{1,2}\rho_{2,2}&=\rho_{1,2}B_{1,2}^{-1} \rho_{1,2}B_{1,2}\rho_{1,2}^{-1}\\
\rho_{2,1}^{-1}B_{1,2}\rho_{2,1}&=\rho_{1,1}B_{1,2}\rho_{1,1}^{-1} & 
\rho_{2,2}^{-1}B_{1,2}\rho_{2,2}&=\rho_{1,2}B_{1,2}\rho_{1,2}^{-1}.
\end{align*}
As in the proof of \repr{presP2T1}, these equalities may be derived geometrically.
\end{rem}

The presentation of $P_2(\T\setminus\{1\}, (x_1,x_2))$ given by \repr{presP2T1} may be modified to obtain another presentation that highlights its algebraic structure as a semi-direct product of free groups of finite rank.

\begin{prop}\label{prop:presTminus1alta}
The group $P_2(\T\setminus\{1\}, (x_1,x_2))$ admits the following presentation:
\begin{enumerate}
\item[generators:] $B$, $u$, $v$, $x$ and $y$.
\item[relations:]\mbox{}
\begin{enumerate}[(a)]
\item\label{it:altpresa} $xux^{-1}=u$.
\item\label{it:altpresb} $xvx^{-1}=v[v^{-1}, u]B^{-1}[u, v^{-1}]$.
\item\label{it:altpresc} $xBx^{-1}=u[v^{-1}, u]B[u, v^{-1}]u^{-1}$.
\item\label{it:altpresd} $yuy^{-1}=v[v^{-1}, u]Buv^{-1}$.
\item\label{it:altprese} $yvy^{-1}=v$.
\item\label{it:altpresf} $yBy^{-1}=v[v^{-1},u]B[u,v^{-1}]v^{-1}=uvu^{-1}Buv^{-1}u^{-1}$.
\end{enumerate}
\end{enumerate}
In particular, $P_2(\T\setminus\{1\}, (x_1,x_2))$ is a semi-direct product of the free group of rank three generated by $\brak{u,v,B}$ by the free group of rank two generated by $\brak{x,y}$.
\end{prop}

\begin{proof}
Using relation~(\ref{it:relh}) of \repr{presP2T1}, we define $B'$ as $B_{1,2}^{-1} [\rho_{2,1},\rho_{2,2}^{-1}]$ and $B_{1,2}$ as $B^{-1} [\rho_{1,1},\rho_{1,2}^{-1}]$. We then apply the following change of variables:
\begin{equation}\label{eq:uvxy}
\text{$u=\rho_{1,1}$,  $v=\rho_{1,2}$, $x=\rho_{1,1}B_{1,2}^{-1}\rho_{2,1}$ and $y=\rho_{1,2}B_{1,2}^{-1}\rho_{2,2}$.}
\end{equation}
Relations~(\ref{it:presP2T1a})--(\ref{it:presP2T1e}) of \repr{presP2T1} may be seen to give rise to relations~(\ref{it:altpresa})--(\ref{it:altpresf}) of \repr{presTminus1alta}. Rewritten in terms of the generators of \repr{presTminus1alta}, relations~(\ref{it:presP2T1f})--(\ref{it:relg}) of \repr{presP2T1} are consequences of relations~(\ref{it:altpresa})--(\ref{it:altpresf}) of \repr{presTminus1alta}. To see this, using the relations of \repr{presTminus1alta}, first note that:
\begin{align*}
xB_{1,2}x^{-1}&= x B^{-1} [u,v^{-1}] x^{-1}=u[v^{-1}, u]B^{-1} [u, v^{-1}]u^{-1} \ldotp \bigl[u, [v^{-1}, u] B[u, v^{-1}] v^{-1}\bigr]\\
&= B^{-1} [u,v^{-1}] = B_{1,2}\;\text{and}\\
yB_{1,2}y^{-1} &= y B^{-1} [u,v^{-1}] y^{-1} = v[v^{-1},u]B^{-1}[u,v^{-1}]v^{-1}\ldotp \bigl[ v[v^{-1}, u]Buv^{-1}, v^{-1}\bigr]\\
&= B^{-1} [u,v^{-1}] = B_{1,2}.
\end{align*}
In light of these relations, it is convenient to carry out the calculations using $B_{1,2}$ instead of $B$. 
In conjunction with the relations of the preceding remark, we obtain the following relations:
\begin{equation}\label{eq:conjxyuv}
\left\{
\begin{aligned}
yuy^{-1}&=vB_{1,2}^{-1}uv^{-1},\; xvx^{-1}=uvu^{-1} B_{1,2}\\
y^{-1}uy&=B_{1,2} v^{-1}uv,\; x^{-1}vx=u^{-1}v B_{1,2}^{-1}u,
\end{aligned}
\right.
\end{equation}
from which it follows that:
\begin{align*}
\rho_{2,1}B_{1,2}\rho_{2,1}^{-1}&=B_{1,2}u^{-1}x B_{1,2} x^{-1}u B_{1,2}^{-1}= B_{1,2}u^{-1} B_{1,2}u B_{1,2}^{-1}=B_{1,2} \rho_{1,1}^{-1}B_{1,2} \rho_{1,1}B_{1,2}^{-1}\;\text{and}\\
\rho_{2,2}B_{1,2}\rho_{2,2}^{-1}&=B_{1,2}v^{-1}y B_{1,2} y^{-1}v B_{1,2}^{-1}=B_{1,2}v^{-1} B_{1,2} v B_{1,2}^{-1}= B_{1,2} \rho_{1,2}^{-1}B_{1,2} \rho_{1,2}B_{1,2}^{-1}.
\end{align*}
Thus relations~(\ref{it:altpresa})--(\ref{it:altpresf}) of \repr{presTminus1alta} imply relations~(\ref{it:presP2T1f}) of \repr{presP2T1}. Now $B'= B_{1,2}^{-1}[\rho_{2,1},\rho_{2,2}^{-1}]$, and a straightforward computation shows that:
\begin{align*}
B'=u^{-1}x  y^{-1}vB_{1,2}^{-1}  x^{-1}u v^{-1}y&=u^{-1}x  y^{-1}vB_{1,2}^{-1}  x^{-1}u v^{-1} \ldotp xyx^{-1}\ldotp xy^{-1}x^{-1}y\\
&=[v^{-1},u]B_{1,2}\ldotp xy^{-1}x^{-1}y.
\end{align*}
One may then check that:
\begin{align*}
xy^{-1}x^{-1}y u y^{-1}xyx^{-1}&=B_{1,2}^{-1}[u,v^{-1}] u[v^{-1}, u] B_{1,2}\;\text{and}\\
xy^{-1}x^{-1}y v y^{-1}xyx^{-1}&=B_{1,2}^{-1}[u,v^{-1}] v [v^{-1}, u] B_{1,2},
\end{align*}
and that:
\begin{align*}
B' \rho_{1,1}B'^{-1}&=B' uB'^{-1}=[v^{-1},u] B_{1,2}\ldotp xy^{-1}x^{-1}y u y^{-1}xyx^{-1} B_{1,2}^{-1}  [u,v^{-1}]=u=\rho_{1,1}\; \text{and}\\
B' \rho_{1,2} B'^{-1}&= B' vB'^{-1}=[v^{-1},u] B_{1,2}\ldotp xy^{-1}x^{-1}y v y^{-1}xyx^{-1} B_{1,2}^{-1}  [u,v^{-1}]=v=\rho_{1,2}.
\end{align*}
Hence relations~(\ref{it:altpresa})--(\ref{it:altpresf}) of \repr{presTminus1alta} imply relations~(\ref{it:relf}) of \repr{presP2T1}. Furthermore,
\begin{align*}
B' B_{1,2}B'^{-1} &= v^{-1}uvu^{-1} B_{1,2} xy^{-1}x^{-1}y B_{1,2} y^{-1}xyx^{-1} B_{1,2}^{-1}  uv^{-1}u^{-1}v\\
&= v^{-1}uvu^{-1} B_{1,2} uv^{-1}u^{-1}v=B_{1,2}^{-1}\ldotp B_{1,2} [v^{-1},u] B_{1,2} [u,v^{-1}] B_{1,2}^{-1} \ldotp B_{1,2}\\
&= B_{1,2}^{-1} B^{-1} B_{1,2} B B_{1,2}, 
\end{align*}
and since
\begin{align*}
xy^{-1}x^{-1}y [u , v^{-1}] y^{-1}xyx^{-1} &= \left[B_{1,2}^{-1}[u,v^{-1}] u[v^{-1}, u] B_{1,2}, B_{1,2}^{-1}[u,v^{-1}] v^{-1} [v^{-1}, u] B_{1,2}\right]\\
&= B_{1,2}^{-1}[u,v^{-1}] B_{1,2},
\end{align*}
we obtain:
\begin{align*}
B' B B'^{-1} &=[v^{-1},u]B_{1,2} B_{1,2}^{-1}[u,v^{-1}] B_{1,2}^{-1} B_{1,2} B_{1,2}^{-1}[u,v^{-1}]=B_{1,2}^{-1} [u,v^{-1}] B_{1,2}^{-1} \ldotp B_{1,2}\\
&=B_{1,2}^{-1} B B_{1,2}.
\end{align*}
Hence relations~(\ref{it:altpresa})--(\ref{it:altpresf}) of \repr{presTminus1alta} imply relations~(\ref{it:relg}) of \repr{presP2T1}. This proves the first part of the statement. The last part of the statement is a consequence of the nature of the presentation.
\end{proof}

The homomorphism $\map{\alpha}{P_2(\T\setminus\{1\})}{P_{2}(\T)}$ mentioned in the remark that follows the statement of \repr{presP2T1} may be used to obtain a presentation of $P_{2}(T)$ in terms of the generators of \repr{presTminus1alta}. To do so, we first show that $\ker{\alpha}$ is the normal closure in $P_{2}(\T\setminus\brak{1})$ of $B$ and $B'$. 

\begin{prop}\label{prop:presTminus2alta}
$\ker{\alpha}=\ang{\!\ang{B,B'}\!}_{P_{2}(\T\setminus\brak{1})}$.
\end{prop}

\begin{proof}
Consider the presentation of $P_{2}(\T\setminus\brak{1})$ given in \repr{presP2T1}, as well as the short exact sequence~\reqref{fnses} and the surjective homomorphism $\map{\alpha}{P_2(\T\setminus\{1\})}{P_{2}(\T)}$. In terms of the generators of \repr{presP2T1}, for $i=1,2$, $\alpha$ sends $\rho_{i,j}$ (resp.\ $B_{1,2}$) (considered as an element of $P_{2}(\T\setminus\brak{1}, (x_{1},x_{2}))$) to $\rho_{i,j}$ (resp.\ to $B_{1,2}$) (considered as an element of $P_{2}(\T)$), and it sends $B$ and $B'$ to the trivial element. Since $B,B'\in \ker{\alpha}$, it is clear that $\ang{\!\ang{B,B'}\!}_{P_{2}(\T\setminus\brak{1})}\subset \ker{\alpha}$. We now proceed to prove the converse inclusion. Using the projection $(p_{2})_{\#}$ and $\alpha$, we obtain the following commutative diagram of short exact sequences:
\begin{equation}\label{eq:keralpha}
\begin{gathered}
\begin{tikzcd}[ampersand replacement=\&]
\& 1 \ar{d} \& 1 \ar{d} \& 1 \ar{d} \&\\
1 \ar{r} \& \ker{\overline{\alpha}} \ar{r}  \ar[dashed]{d}{\tau\left\lvert_{\ker{\overline{\alpha}}}\right.} \& P_{1}(\T\setminus\brak{1,x_{2}},x_{1}) \ar[dashed]{r}{\overline{\alpha}} \ar{d}{\tau} \& P_{1}(\T\setminus\brak{x_{2}},x_{1}) \ar{r} \ar{d} \& 1\\
1 \ar{r}  \& \ker{\alpha} \ar{r} \ar[dashed, shift left=1]{d}{(p_{2})_{\#}\left\lvert_{\ker{\alpha}}\right.} \& P_{2}(\T\setminus\brak{1}, (x_{1},x_{2}))  \ar{r}{\alpha} \ar{d}{(p_{2})_{\#}} \& P_{2}(\T, (x_{1},x_{2})) \ar{r} \ar{d}{(p_{2}')_{\#}} \& 1\\
1 \ar{r} \& \ker{\alpha'} \ar{r} \ar{d} \ar[dashed, shift left=1]{u}{s} \& P_{1}(\T\setminus\brak{1}, x_{2})  \ar{r}{\alpha'} \ar{d} \& P_{1}(\T,x_{2}) \ar{r} \ar{d} \& 1.\\
\& 1 \& 1  \& 1  \& 
\end{tikzcd}
\end{gathered}
\end{equation}
Diagram~\reqref{keralpha} is constructed in the following manner: the surjective homomorphism $\map{\alpha'}{P_{1}(\T\setminus\brak{1}, x_{2})}{P_{1}(\T,x_{2})}$ is induced by the inclusion $\map{\iota'}{\T\setminus\brak{1}}{\T}$, and the right-hand column is a Fadell-Neuwirth short exact sequence, where the fibration $\map{p_{2}'}{F_{2}(\T)}{\T}$ is given by projecting onto the second coordinate. The homomorphism $\map{\tau}{P_{1}(\T\setminus\brak{1,x_{2}},x_{1})}{P_{2}(\T\setminus\brak{1}, (x_{1},x_{2}))}$ is interpreted as inclusion. Taking $\brak{\rho_{2,1},\rho_{2,2},B'}$ to be a generating set of $P_{1}(\T\setminus\brak{1}, x_{2})$ subject to the relation $[\rho_{2,1},\rho_{2,2}^{-1}]=B'$, as for $\alpha$, we see that $\alpha'(\rho_{2,j})=\rho_{2,j}$ and $\alpha'(B')=1$. Alternatively, we may consider $P_{1}(\T\setminus\brak{1}, x_{2})$ to be the free group on $\brak{\rho_{2,1},\rho_{2,2}}$, and $P_{1}(\T,x_{2})$ to be the free Abelian group on $\brak{\rho_{2,1},\rho_{2,2}}$, so $\alpha'$ is Abelianisation, and $\ker{\alpha'}=\Gamma_{2}(P_{1}(\T\setminus\brak{1}, x_{2}))$. Interpreting $\alpha'$ as the canonical projection of $P_{1}(\T\setminus\brak{1}, x_{2})$ onto its quotient by the normal closure of the element $[\rho_{2,1},\rho_{2,2}^{-1}]$ in $P_{1}(\T\setminus\brak{1}, x_{2})$, we see that $\ker{\alpha'}=\ang{\!\ang{\,[\rho_{2,1},\rho_{2,2}^{-1}]\,}\!}_{P_{1}(\T\setminus\brak{1}, x_{2})}$. But $[\rho_{2,1},\rho_{2,2}^{-1}]=B'$ in $P_{1}(\T\setminus\brak{1}, x_{2})$, hence:
\begin{equation}\label{eq:keralphaprime}
\ker{\alpha'}=\ang{\!\ang{B'}\!}_{P_{1}(\T\setminus\brak{1},x_{2})}.
\end{equation}
The fact that $P_{1}(\T\setminus\brak{1}, x_{2})$ is a free group implies that $\ker{\alpha'}$ is too (albeit of infinite rank). The commutativity of the lower right-hand square of~\reqref{keralpha} is a consequence of the following commutative square:
\begin{equation*}
\begin{tikzcd}[ampersand replacement=\&]
F_{2}(\T\setminus\brak{1})  \ar[r, "\iota"] \ar[d, "p_{2}"] \& F_{2}(\T) \ar[d, "p_{2}'"]\\
F_{1}(\T\setminus\brak{1})  \ar[r, "\iota'"]  \& F_{2}(\T),
\end{tikzcd}
\end{equation*}
where $\map{\iota}{F_{2}(\T\setminus\brak{1})}{F_{2}(\T)}$ is induced by the inclusion $\iota'$. Together with $\alpha$ and $\alpha'$, the second two columns of~\reqref{keralpha} give rise to a homomorphism $\map{\overline{\alpha}}{P_{1}(\T\setminus\brak{1,x_{2}}, x_{1})}{P_{1}(\T \setminus\brak{x_{2}},x_{1})}$. Taking $\brak{\rho_{1,1},\rho_{1,2},B_{1,2},B}$ to be a generating set of $P_{1}(\T\setminus\brak{1,x_{2}}, x_{1})$, by commutativity of the diagram, we see that $\overline{\alpha}$ sends each of $\rho_{1,1}$, $\rho_{1,2}$ and $B_{1,2}$ (considered as an element of $P_{1}(\T\setminus\brak{1,x_{2}}, (x_{1},x_{2}))$) to itself (considered as an element of $P_{1}(\T\setminus\brak{x_{2}}, x_{1})$), and sends $B'$ to the trivial element. In particular, $\overline{\alpha}$ is surjective. As for $\alpha'$, we may consider $P_{1}(\T\setminus\brak{1,x_{2}}, x_{1})$ to be the free group of rank three generated by $\brak{\rho_{1,1},\rho_{1,2},B_{1,2}}$, and $P_{1}(\T\setminus\brak{x_{2}}, x_{1})$ to be the group generated by $\brak{\rho_{1,1},\rho_{1,2},B_{1,2}}$ subject to the relation $[\rho_{1,1},\rho_{1,2}^{-1}]=B_{1,2}$. The homomorphism $\alpha'$ may thus be interpreted as the canonical projection of $P_{1}(\T\setminus\brak{x_{2}}, x_{1})$ onto its quotient by $\ang{\!\ang{[\rho_{1,1},\rho_{1,2}^{-1}]B_{1,2}^{-1}}\!}_{P_{1}(\T\setminus\brak{x_{2}}, x_{1})}$. But by relation~(\ref{it:relh}) of \repr{presP2T1}, $B'=[\rho_{1,1},\rho_{1,2}^{-1}]B_{1,2}^{-1}$, hence:
\begin{equation*}
\ker{\overline{\alpha}}=\ang{\!\ang{B'}\!}_{P_{1}(\T\setminus\brak{x_{2}},x_{1})}.
\end{equation*}
By exactness, the first (resp.\ second) two rows of~\reqref{keralpha} give rise to an induced homomorphism $\map{\tau\left\lvert_{\ker{\overline{\alpha}}}\right.}{\ker{\overline{\alpha}}}{\ker{\alpha}}$ (resp.\ $\map{(p_{2})_{\#}\left\lvert_{\ker{\alpha}}\right.}{\ker{\alpha}}{\ker{\alpha'}}$), and $\tau\left\lvert_{\ker{\overline{\alpha}}}\right.$ is injective because $\tau$ is. The homomorphism $(p_{2})_{\#}\left\lvert_{\ker{\alpha}}\right.$ is surjective, because by~\reqref{keralphaprime}, any element $x$ of $\ker{\alpha'}$ may be written as a product of conjugates of $B'$ and its inverse by products of $\rho_{2,1}$, $\rho_{2,2}$ and $B'$. This expression, considered as an element of $P_{2}(\T\setminus\brak{1}, (x_{1},x_{2}))$, belongs to $\ker{\alpha}$, and its image under $(p_{2})_{\#}$ is equal to $x$.

The fact that $\im{\tau\left\lvert_{\ker{\overline{\alpha}}}\right.}\subset \ker{(p_{2})_{\#}\left\lvert_{\ker{\alpha}}\right.}$ follows from exactness of the second column of~\reqref{keralpha}. Conversely, if $z\in \ker{(p_{2})_{\#}\left\lvert_{\ker{\alpha}}\right.}$ then $z\in \ker{\tau}$ by exactness of the second column. So there exists $y\in P_{1}(\T\setminus\brak{1,x_{2}},x_{1})$ such that $\tau(y)=z$. But $\overline{\alpha}(y)=\alpha(\tau(y))=\alpha(z)=1$, and hence $y\in \ker{\overline{\alpha}}$. This proves that $\ker{(p_{2})_{\#}\left\lvert_{\ker{\alpha}}\right.}\subset \im{\tau\left\lvert_{\ker{\overline{\alpha}}}\right.}$, and we deduce that the first column is exact.

Finally, since $\ker{\alpha'}$ is free, we may pick an infinite basis consisting of conjugates of $B'$ by certain elements of $P_{1}(\T\setminus\brak{1}, x_{2})$. Further, there exists a section $\map{s}{\ker{\alpha'}}[\ker{\alpha}]$ for $(p_{2})_{\#}\left\lvert_{\ker{\alpha}}\right.$ that consists in sending each of these conjugates (considered as an element of $\ker{\alpha'}$) to itself (considered as an element of $\ker{\alpha}$). In particular, $\ker{\alpha}$ is an internal semi-direct product of $\tau(\ker{\overline{\alpha}})$, which is contained in $\ang{\!\ang{B}\!}_{P_{2}(\T\setminus\brak{1})}$, by $s(\ker{\overline{\alpha}})$, which is contained in $\ang{\!\ang{B'}\!}_{P_{2}(\T\setminus\brak{1})}$. Hence $\ker{\alpha}$ is contained in the subgroup $\ang{\!\ang{B,B'}\!}_{P_{2}(\T\setminus\brak{1})}$. This completes the proof of the proposition.
\end{proof}

We may thus deduce the following useful presentation of $P_{2}(\T)$.

\begin{cor}\label{cor:compactpres}
The group $P_2(\T, (x_1,x_2))$ admits the following presentation:
\begin{enumerate}
\item[generators:] $u$, $v$, $x$ and $y$.
\item[relations:]\mbox{}
\begin{enumerate}[(a)]
\item\label{it:altpresA} $xux^{-1}=u$ and $yuy^{-1}=u$.
\item\label{it:altpresB} $xvx^{-1}=v$ and $yvy^{-1}=v$.
\item\label{it:altpresC} $xyx^{-1}=y$. 
\end{enumerate}
\end{enumerate}
In particular, $P_2(\T, (x_1,x_2))$ is isomorphic to the direct product of the free group of rank two generated by $\brak{u,v}$ and the free Abelian group of rank two generated by $\brak{x,y}$.
\end{cor}

\begin{rem}
The decomposition of \reco{compactpres} is a special case of~\cite[Lemma~17]{BGG}.
\end{rem}

\begin{proof}[Proof of \reco{compactpres}]
By \repr{presTminus2alta}, a presentation of $P_2(\T)$ may be obtained from the presentation of $P_2(\T\setminus\{1\})$ given in \repr{presTminus1alta} by setting $B$ and $B'$ equal to $1$. Under this operation, relations~(\ref{it:altpresa}) and~(\ref{it:altpresd}) (resp.~(\ref{it:altpresb}) and~(\ref{it:altprese})) of \repr{presTminus1alta} are sent to relations~(\ref{it:altpresA}) (resp.~(\ref{it:altpresB})) of \reco{compactpres}, and relations~(\ref{it:altpresc}) and~(\ref{it:altpresf}) of \repr{presTminus1alta} become trivial. We must also take into account the fact that $B'=1$ in $P_2(\T)$. In $P_2(\T\setminus\{1\})$, we have:
\begin{equation*}
B' = B_{1,2}^{-1} [\rho_{2,1}, \rho_{2,2}^{-1}]= [v^{-1},u] B \bigl[B^{-1}[u,v^{-1}]\ldotp u^{-1}x, y^{-1}v  [v^{-1},u] B\bigr],
\end{equation*}
and taking the image of this equation by $\alpha$, we obtain:
\begin{align*}
1 &=  [v^{-1},u] \bigl[[u,v^{-1}] u^{-1}x, y^{-1}v  [v^{-1},u] \bigr]\\
& =  [v^{-1},u] \ldotp [u,v^{-1}] u^{-1}x \ldotp y^{-1}v  [v^{-1},u] \ldotp x^{-1}u  [v^{-1},u] \ldotp  [u,v^{-1}] v^{-1} y =  u^{-1}x y^{-1}uvu^{-1} x^{-1}u v^{-1} y
\end{align*}
in $P_2(\T)$. Using relations~(\ref{it:altpresA}) (resp.~(\ref{it:altpresB})) of \reco{compactpres}, it follows that $[x, y^{-1}]=1$, which yields relation~(\ref{it:altpresC}) of \reco{compactpres}. The last part of the statement is a consequence of the nature of the presentation.
\end{proof}

\begin{rem}
As we saw in \repr{presTminus1alta}, $P_2(\T\setminus\{1\}, (x_1, x_2))$ is a semi-direct product of the form $\mathbb{F}_3(u,v,B)\rtimes \mathbb{F}_2(x,y)$, the action being given by the relations of that proposition. Transposing the second two columns of the commutative diagram~\reqref{keralpha}, we obtain, up to isomorphism, the following commutative diagram: 
\begin{equation}\label{eq:f3f2}
\begin{tikzcd}[ampersand replacement=\&]
1 \ar{r}\& \mathbb{F}_3(u,v,B) \ar{d}{\overline{\alpha}} \ar{r}{\tau} \& \mathbb{F}_3(u,v,B)\rtimes \mathbb{F}_2(x,y)  \ar{d}{\alpha} \ar{r}{(p_{2})_{\#}} \ar{r} \&  \mathbb{F}_2(x,y)  \ar{d}{\alpha'} \ar{r} \& 1\\
1 \ar{r}\& \mathbb{F}_2(u,v) \ar{r} \& \mathbb{F}_2(u,v) \times \Z^{2} \ar{r}{(p_{2}')_{\#}} \& \Z^{2} \ar{r} \& 1,
\end{tikzcd}
\end{equation}
where $\alpha(u)=u$, $\alpha(v)=v$, $\alpha(B)=1$, $\alpha(x)=(1; (1,0))$ and $\alpha(y)=(1; (0,1))$. This is a convenient setting to study the question of whether a $2$-valued map may be deformed to a root-free $2$-valued map (which implies that the corresponding map may be deformed to a fixed point free map), since using~\reqref{f3f2} and \reth{defchinegI}(\ref{it:defchinegbI}), the question is equivalent to a lifting problem, to which we will refer in \resec{fptsplit2}, notably in Propositions~\ref{prop:necrootfree2}  and~\ref{prop:construct2valprop}.    
\end{rem}
 
Using the short exact sequence~\reqref{sesbraid}, we obtain the following presentation of the full braid group $B_2(\T, (x_1,x_2))$ from that of $P_2(\T, (x_1,x_2))$ given by \reco{compactpres}.

\begin{prop}\label{prop:presful}
The group $B_2(\T, (x_1,x_2))$ admits the following presentation:
\begin{enumerate}
\item[generators:] $u$, $v$, $x$, $y$ and $\sigma$.
\item[relations:]\mbox{}
\begin{enumerate}[(a)]
\item\label{it:presB2Ta} $xux^{-1}=u$ and $yuy^{-1}=u$.
\item $xvx^{-1}=v$ and $yvy^{-1}=v$.
\item\label{it:presB2Tc} $xyx^{-1}=y$. 
\item\label{it:presB2Td} $\sigma^{2}=[u,v^{-1}]$.
\item\label{it:presB2Te} $\sigma x \sigma^{-1}=x$ and $\sigma y \sigma^{-1}=y$.
\item\label{it:presB2Tf} $\sigma u \sigma^{-1}=[u,v^{-1}]u^{-1}x$ and $\sigma v \sigma^{-1}=[u,v^{-1}]v^{-1}y$.
\end{enumerate}
\end{enumerate}
\end{prop} 

\begin{proof}
Once more, we apply the methods of~\cite[Proposition~1, p.~139]{Jo}, this time to the short exact sequence~\reqref{sesbraid} for $X=\T$ and $n=2$, where we take $P_{2}(\T)$ to have the presentation given by \reco{compactpres}. A coset representative of the generator of $\Z_{2}$ is given by the braid $\sigma=\sigma_{1}$ that swaps the two basepoints. Hence $\brak{u,v,x,y,\sigma}$ generates $B_{2}(\T)$. Relations~(\ref{it:presB2Ta})--(\ref{it:presB2Tc}) emanate from the relations of $P_{2}(\T)$. Relation~(\ref{it:presB2Td}) is obtained by lifting the relation of $\Z_{2}$ to $B_{2}(\T)$ and using the fact that $\sigma^{2}=B_{1,2}=[u,v^{-1}]$. To obtain relations~(\ref{it:presB2Te}) and~(\ref{it:presB2Tf}), by geometric arguments, one may see that for $j\in \brak{1,2}$, $\sigma \rho_{1,j} \sigma^{-1}= \rho_{2,j}$ and $\sigma \rho_{2,j} \sigma^{-1}= B_{1,2}\rho_{2,j}B_{1,2}^{-1}$, and one then uses \req{uvxy} to express these relations in terms of $u,v,x$ and $y$.
\end{proof}

\subsection{A description of the homotopy classes of $2$-ordered and split $2$-valued maps of $\T$, and the computation of the Nielsen number}\label{sec:descript}

In this section, we describe the homotopy classes of $2$-ordered (resp.\ split $2$-valued) maps of $\T$ using the group structure of $P_{2}(\T)$ (resp.\ of $B_{2}(\T)$) given in \resec{toro2}. 

\begin{prop}\label{prop:baseT2}\mbox{}
\begin{enumerate}
\item\label{it:baseT2a} The set $[\T,F_{2}(\T)]_0$  of based homotopy classes of $2$-ordered  maps of $\T$ is in one-to-one correspondence with the set of commuting, ordered pairs of elements of $P_{2}(\T)$.
\item\label{it:baseT2b} The set $[\T,F_{2}(\T)]$ of homotopy classes of $2$-ordered  maps of $\T$ is in one-to-one correspondence with the set of commuting, conjugate, ordered pairs of $P_{2}(\T)$, i.e.\ two  commuting pairs $(\alpha_1, \beta_1)$ and $(\alpha_2, \beta_2)$ of $P_{2}(\T)$ give rise to the same homotopy class of $2$-ordered maps of $\T$ if there exists  $\delta\in P_{2}(\T)$ such that $\delta\alpha_1\delta^{-1}=\alpha_2$ and $\delta\beta_1\delta^{-1}=\beta_2$.
\item\label{it:baseT2c} Under the projection $\widehat{\pi}\colon\thinspace [\T,F_{2}(\T)] \to [\T,D_{2}(\T)]$ induced by the covering map $\pi\colon\thinspace F_2(\T) \to D_2(\T)$, two homotopy classes of $2$-ordered maps of $\T$ are sent to the same homotopy class of $2$-unordered maps if and only if any two pairs of braids that represent the maps are conjugate in $B_{2}(\T)$. 
\end{enumerate} 
\end{prop}
 
\begin{proof}  Let $x_0\in \T$ and $(y_0, z_0)\in F_2(\T)$ be basepoints, and let $\Psi\colon\thinspace  \T \to F_{2}(\T)$ be a basepoint-preserving $2$-ordered map. The restriction of $\Psi$ to the meridian $\mu$ and the longitude $\lambda$ of $\T$, which are geometric representatives of the elements of the basis $(e_{1},e_{2})$ of $\pi_1(\T)$, gives rise to a pair of geometric braids. The resulting pair $(\Psi_{\#}(e_{1}),\Psi_{\#}(e_{2}))$ of elements of $P_{2}(\T)$ obtained via the induced homomorphism $\Psi_{\#}\colon\thinspace \pi_{1}(\T) \to P_{2}(\T)$ is an invariant of the based homotopy class of the map $\Psi$, and the two braids $\Psi_{\#}(e_{1})$ and $\Psi_{\#}(e_{2})$ commute. Conversely, given a pair of braids $(\alpha, \beta)$ of $P_{2}(\T)$,  let $f_{1}\colon\thinspace  \St[1] \to F_{2}(\T)$ and $f_{2}\colon\thinspace  \St[1] \to F_{2}(\T)$ be geometric representatives of $\alpha$ and $\beta$ respectively, \emph{i.e.}\ $\alpha=[f_{1}]$ and $\beta=[f_{2}]$. Then we define a geometric map from the wedge of two circles into $F_{2}(\T)$ by sending $x\in \St[1] \vee \St[1]$ to $f_{1}(x)$ (resp.\ to $f_{2}(x)$) if $x$ belongs to the first (resp.\ second) copy of $\St[1]$. By classical obstruction theory in low dimension, this map extends to $\T$ if and only if $\alpha$ and $\beta$ commute as elements of $P_{2}(\T)$, and part~(\ref{it:baseT2a}) follows.  Parts~(\ref{it:baseT2b}) and~(\ref{it:baseT2c}) are consequences of  part~(\ref{it:baseT2a}) and classical general facts about maps between spaces of type $K(\pi, 1)$, see~\cite[Chapter~V, Theorem~4.3]{Wh} for example.  
\end{proof}
 
Applying \repr{equivsplit} to $\T$, we obtain the following consequence.

\begin{prop}\label{prop:lifttorus}
If $\phi\colon\thinspace  \T \multimap \T$ is a split $2$-valued map and $\widehat{\Phi}\colon\thinspace \T \to F_{2}(\T)$ is a lift of $\phi$, the map $\phi$ can be deformed to a fixed point free $2$-valued map if and only if there exist commuting elements $\alpha_1, \alpha_2 \in P_{3}(\T)$ such that $\alpha_1$ (resp.\ $\alpha_2$) projects to $(e_1, \widehat{\Phi}_{\#}(e_1))$ (resp.\ $(e_2, \widehat{\Phi}_{\#}(e_2))$) under the homomorphism induced by the inclusion map $\widehat{\iota}_{3}\colon\thinspace F_{3}(\T) \to \T\times F_{2}(\T)$.
\end{prop}

\begin{proof}
Since $\T$ is a space of type $K(\pi, 1)$, the existence of diagram~\reqref{commdiag4} is equivalent to that of the corresponding induced diagram on the level of fundamental groups. It then suffices to take $\alpha_{1}=\Theta_{\#}(e_1)$ and $\alpha_{2}=\Theta_{\#}(e_2)$ in the statement of \repr{equivsplit}.
\end{proof}

Proposition~\ref{prop:lifttorus} gives a criterion to decide whether a split $2$-valued map of $\T$ can be deformed to a fixed point free $2$-valued map. However, from a computational point of view, it seems better to use an alternative condition in terms of roots (see \resec{exrfm}).

In the following proposition, we make use of the identification of $P_{2}(\T)$ with $\mathbb{F}_{2} \times \Z^{2}$ given in \reco{compactpres}.

\begin{prop}\label{prop:exismaps}\mbox{}
\begin{enumerate}
\item\label{it:exismapsa} The set $[\T,F_{2}(\T)]_0$ of  based  homotopy classes of $2$-ordered maps of $\T$ is in one-to-one correspondence with the set of pairs $(\alpha, \beta)$ of elements of $\mathbb{F}_2 \times \Z^{2}$ of the form $\alpha=(w^r,(a,b))$, $\beta=(w^s, (c,d))$, where $(a,b), (c,d), (r,s)\in \Z^{2}$ and $w\in \mathbb{F}_2$. Further, up to taking a root of $w$ if necessary, we may assume that $w$ is either trivial or is a primitive element of $\mathbb{F}_2$ (i.e.\ $w$ is not a proper power of another element of $\mathbb{F}_2$).
 \item\label{it:exismapsb} The set  $[\T,F_{2}(\T)]$ of homotopy classes of $2$-ordered maps of $\T$  is in one-to-one correspondence with the set of the equivalence classes of pairs $(\alpha, \beta)$ of elements of $\mathbb{F}_2 \times \Z^{2}$ of the form given in part~(\ref{it:exismapsa}), where the equivalence relation is defined as follows: the pairs of elements $( (w_1^{r_{1}},(a_1,b_1)), (w_1^{s_{1}}, (c_1,d_1)))$ and $( (w_2^{r_{2}},(a_2,b_2)), (w_2^{s_{2}}, (c_2,d_2)))$ of $\mathbb{F}_2 \times \Z^{2}$ are equivalent if and only if $(a_1,b_1,c_1,d_1) = (a_2,b_2,c_2,d_2)$, and either:
  \begin{enumerate}[(i)]
\item $w_1=w_2=1$, or 
\item $w_1$ and $w_2$ are primitive, and there exists $\epsilon\in \brak{1,-1}$ such that $w_1$ and $w_2^{\epsilon}$ are conjugate in $\mathbb{F}_2$, 
and $(r_1,s_{1})=\epsilon (r_2,s_{2})\ne (0,0)$.
\end{enumerate}
\end{enumerate}
\end{prop}

\begin{proof}
Part~(\ref{it:exismapsa}) follows using \repr{baseT2}(\ref{it:baseT2a}), the identification of $P_{2}(\T)$ with $\mathbb{F}_2\times \Z^2$ given by \reco{compactpres}, and the fact that two elements of $\mathbb{F}_2$ commute if and only if they are powers of some common element of $\mathbb{F}_2$. Part~(\ref{it:exismapsb}) is a consequence of \repr{baseT2}(\ref{it:baseT2b}), \reco{compactpres}, and the straightforward description of the conjugacy classes of the group $\mathbb{F}_2\times \Z^2$. 
\end{proof}

To describe the  homotopy classes of split $2$-valued maps of $\T$, let us consider the set  $[\T,F_{2}(\T)]$ of  homotopy classes of $2$-ordered maps  and the action of $\Z_2$ on this set that is induced by the action of $\Z_2$ on $F_2(\T)$.  By \relem{split}(\ref{it:splitIII}), the corresponding set of orbits $[\T,F_{2}(\T)]/\Z_{2}$ is in one-to-one correspondence with the set $\splitmap{\T}{\T}{2}/\!\sim$ of homotopy classes of split $2$-valued maps of $\T$. 
Given a homotopy class of a $2$-ordered map of $\T$, choose a based representative $f\colon\thinspace  \T \to F_2(\T)$. The based homotopy class of $f$ is determined by the element $f_{\#}$ of $\operatorname{\text{Hom}}(\Z^{2}, P_2(\T))$. In turn, by \repr{exismaps}(\ref{it:exismapsa}), 
$f_{\#}$ is determined by a pair of elements of $P_{2}(\T)$ of the form $( (w^r,(a,b)), (w^s, (c,d)))$, where $(a,b)$, $(c,d)$ and $(r,s)$ belong to $\Z^2$, and $w\in \mathbb{F}_2$. To characterise the equivalence class of $f$ in $[\T,F_{2}(\T)]/\Z_{2}$, we first consider the set of conjugates of this pair by the elements of $P_2(\T)$, which by \repr{baseT2}(\ref{it:baseT2b}) describes the homotopy class of $f$ in $[\T,F_{2}(\T)]$, and secondly, we take into account the $\Z_2$-action by conjugating by the elements of $B_2(\T)$. So the equivalence class of $f$ in $[\T,F_{2}(\T)]/\Z_{2}$ is characterised by the set of conjugates of the pair $( (w^r,(a,b)), (w^s, (c,d)))$ by elements of $B_2(\T)$. The presentation of $B_{2}(\T)$ given by~\repr{presful} contains the action by conjugation of $\sigma$ on $P_{2}(\T)$. Consider the homomorphism involution of $\mathbb{F}_2(u,v)$ that is defined on the generators of $\mathbb{F}_2(u,v)$ by $u \longmapsto u^{-1}$ and $v \longmapsto v^{-1}$. The image of an element $w\in \mathbb{F}_2(u,v)$ by this automorphism will be denoted by $\widehat{w}$. With respect to the decomposition of $P_{2}(\T)$ given by \reco{compactpres}, let $\gamma\colon\thinspace  P_2(\T) \to \mathbb{F}_2(u,v)$ denote the projection onto $\mathbb{F}_2(u,v)$. Let $\operatorname{\text{Ab}}\colon\thinspace  \mathbb{F}_2(u,v) \to \Z^{2}$ denote the Abelianisation homomorphism that sends $u$ to $(1,0)$ and $v$ to $(0,1)$. We write $\operatorname{\text{Ab}}(w)=(|w|_{u}, |w|_{v})$, where $|w|_{u}$ (resp.\ $|w|_{v}$) denotes the exponent sum of $u$ (resp.\ $v$) in the word $w$, and $\ell(w)$ will denote the word length of $w$ with respect to $u$ and $v$.  One may check easily that $(\widehat{w})^{-1}=\widehat{w^{-1}}$ and $\ell(w)=\ell(\widehat{w})$ for all $w\in \mathbb{F}_2(u,v)$. Note also that if $\lambda\in \mathbb{F}_2(u,v)$ and $r\in \Z$ then:
\begin{equation}\label{eq:lambdahat}
\widehat{\lambda} (\lambda \widehat{\lambda})^{r} \lambda=\widehat{\lambda} (\lambda \widehat{\lambda})^{r} \widehat{\lambda}^{-1} \ldotp \widehat{\lambda} \lambda=(\widehat{\lambda} \lambda \widehat{\lambda}\widehat{\lambda}^{-1})^{r} \widehat{\lambda} \lambda=(\widehat{\lambda}\lambda)^{r+1}.
\end{equation}

\begin{lem}\label{lem:conjhat}
For all $w\in \mathbb{F}_{2}(u,v)$, $\widehat{w}=vu^{-1}\gamma(\sigma w\sigma^{-1})uv^{-1}$ and 
\begin{equation}\label{eq:conjsigma}
\sigma w \sigma^{-1}=(uv^{-1}\widehat{w} vu^{-1}, (|w|_{u}, |w|_{v})).
\end{equation}
In particular, $\widehat{w}$ is conjugate to $\gamma(\sigma w\sigma^{-1})$ in $\mathbb{F}_2(u,v)$. 
\end{lem}

\begin{proof}
By \repr{presful},
we have:
\begin{equation*}
\text{$\sigma u \sigma^{-1} = [u,v^{-1}]u^{-1} x=uv^{-1}\widehat{u} vu^{-1} x^{|u|_{u}}$ and 
$\sigma v \sigma^{-1} = [u,v^{-1}]v^{-1} y=uv^{-1}\widehat{v} vu^{-1} y^{|v|_{v}}$.}
\end{equation*}
If $w\in \mathbb{F}_{2}(u,v)$ then~\reqref{conjsigma} follows because
$x$ and $y$ belong to the centre of $P_{2}(\T)$. Thus $\gamma(\sigma w \sigma^{-1})=uv^{-1}\widehat{w} vu^{-1}$ as required.
\end{proof}

\begin{lem}\label{lem:conjhat1}\mbox{}
\begin{enumerate}
\item\label{it:classifI} Let $a,b\in \mathbb{F}_{2}(u,v)$ be such that $ab$ is written in reduced form. Then   $ab=\widehat{b}\,\widehat{a}$  if and only if  there exist $\lambda\in \mathbb{F}_{2}(u,v)$ and $r,s\in \Z$ such that:
\begin{equation}\label{eq:concl1}
\text{$a=(\widehat{\lambda} \lambda)^s\widehat{\lambda}$ and $b=(\lambda\widehat{\lambda})^r\lambda$.}
\end{equation}

\item\label{it:classifII} For all $w\in \mathbb{F}_{2}(u,v)$, $w$ and $\widehat{w}$ are conjugate in $\mathbb{F}_2(u,v)$ if and only if there exist $\lambda\in \mathbb{F}_2(u,v)$ and $l\in \Z$ such that:
\begin{equation}\label{eq:concl2}
w=(\lambda \widehat{\lambda})^l.
\end{equation}
\end{enumerate}
\end{lem}

\begin{rem}
By modifying the definition of the $\widehat{\cdot}$ homomorphism appropriately, \relem{conjhat1} and its proof may be generalised to any free group of finite rank on a given set.
\end{rem}

%

\begin{proof}[Proof of \relem{conjhat1}]
We first prove the `if' implications of~(\ref{it:classifI}) and~(\ref{it:classifII}). For part~(\ref{it:classifI}), let $a,b\in \mathbb{F}_{2}(u,v)$ be such that $ab$ is written in reduced form, and that~\reqref{concl1} holds. Using~\reqref{lambdahat}, we have:
\begin{equation*}
ab=(\widehat{\lambda} \lambda)^s\widehat{\lambda}(\lambda\widehat{\lambda})^r\lambda=(\widehat{\lambda} \lambda)^{r+s+1}=(\widehat{\lambda} \lambda)^{r}\widehat{\lambda}(\lambda \widehat{\lambda})^{s}\lambda=
\widehat b\,\widehat a.
\end{equation*}
For part~(\ref{it:classifII}), if~\reqref{concl2} holds then $\widehat{w}=(\widehat{\lambda} \lambda)^l=\widehat{\lambda}(\lambda\widehat{\lambda})^l(\widehat{\lambda})^{-1}=\widehat{\lambda} w\widehat{\lambda}^{-1}$ by~\reqref{lambdahat}, so $w$ and $\widehat{w}$ are conjugate in $\mathbb{F}_2(u,v)$.

Finally, we prove the `only if' implications of~(\ref{it:classifI}) and~(\ref{it:classifII}) simultaneously by induction on the length $k$ of the words $ab$ and $w$, which we assume to be non trivial and written in reduced form. Let~(E1) denote the equation $ab=\widehat{b}\,\widehat{a}$, and let~(E2) denote the equation $\widehat{w}=\theta w \theta^{-1}$, where $\theta\in \mathbb{F}_2(u,v)$. Note that if $a$ and $b$ satisfy~(E1) then both $a$ and $b$ are non trivial, and that $\widehat{b}\,\widehat{a}$ is also in reduced form. Further, if $z\in \mathbb{F}_2(u,v)$, then since $|z|_{y}=-|\widehat{z}|_{y}$ for $y\in\brak{u,v}$, it follows from the form of~(E1) and~(E2) that $k$ must be even in both cases. We carry out the proof of the two implications by induction as follows.

\begin{enumerate}[(i)]
\item If $k\leq 4$ then (E1) implies~\reqref{concl1}. One may check easily that $k$ cannot be equal to $2$. Suppose that $k=4$ and that~(E1) holds. Now $\ell(a)\neq 1$, for otherwise $b$ would start with $a^{-1}$, but then $ab$ would not be reduced. Similarly, $\ell(b)\neq 1$, so we must have $\ell(a)=\ell(b)=2$, in which case $b=\widehat{a}$, and it suffices to take $r=s=0$ and $\lambda=b$ in~\reqref{concl1}. 

\item If $k\leq 4$ then (E2) implies~\reqref{concl2}. Once more, it is straightforward to see that $k$ cannot be equal to $2$. Suppose that $k=4$. Since $w$ and $\widehat{w}$ are conjugate, we have $|w|_{y}=|\widehat{w}|_{y}$ for $y\in\brak{u,v}$, and so $|w|_{y}=0$. Since $w$ is in reduced form, one may then check that~\reqref{concl2} holds, where $\ell(\lambda)=2$.

\item\label{it:induct} Suppose by induction that for some $k\geq 4$,~(E1) implies~\reqref{concl1} if $\ell(ab)<k$ and~(E2) implies~\reqref{concl2} if $\ell(w)<k$. Suppose that $a$ and $b$ satisfy~(E1) and that $\ell(ab)=k$. If $\ell(a)=\ell(b)$ then $b=\widehat{a}$, and as above, it suffices to take $r=s=0$ and $\lambda=b$ in~\reqref{concl1}. So assume that $\ell(a)\neq \ell(b)$. By applying the automorphism $\widehat{\cdot}$ and exchanging the r\^{o}les of $a$ and $b$ if necessary, we may suppose that $\ell(a)<\ell(b)$. We consider two subcases.

\begin{enumerate}[(A)]
\item $\ell(a)\leq \ell(b)/2$: since both sides of~(E1) are in reduced form, $b$ starts and ends with $\widehat{a}$, and there exists $b_{1}\in \mathbb{F}_2(u,v)$ such that $b=\widehat{a}b_{1}\widehat{a}$, written in reduced form. Substituting this into~(E1), we see that $a\widehat{a}b_{1}\widehat{a}= a\widehat{b}_{1}a\widehat{a}$, and thus $\widehat{a}b_{1}=\widehat{b}_{1}a$, written in reduced form. This equation is of the form~(E1), and since $\ell(\widehat{a}b_{1})<\ell(b)<\ell(ab)$, we may apply the induction hypothesis. Thus there exist $\lambda\in \mathbb{F}_{2}(u,v)$ and $r,s\in \Z$ such that 
$\widehat{a}=(\widehat{\lambda} \lambda)^s\widehat{\lambda}$, and $b_1=(\lambda\widehat{\lambda})^{r}\lambda$.  
Therefore $a=(\lambda\widehat{\lambda} )^s \lambda$ and 
\begin{equation*}
b=\widehat{a} b_1 \widehat{a}=(\widehat{\lambda} \lambda)^s\widehat{\lambda}  (\lambda\widehat{\lambda})^{r}\lambda (\widehat{\lambda} \lambda)^s\widehat{\lambda}= (\widehat{\lambda} \lambda)^{2s+r+1}\widehat{\lambda},
\end{equation*} 
using~\reqref{lambdahat}, which proves the result in this case.

\item $\ell(b)/2< \ell(a) <\ell(b)$: since both sides of~(E1) are in reduced form, $b$ starts with $\widehat{a}$, and there exists $b_{1}\in \mathbb{F}_2(u,v)$, $b_{1}\neq 1$, such that $b=\widehat{a}b_{1}$. Substituting this into~(E1), we obtain $a\widehat{a}b_{1}=a\widehat{b}_{1}\widehat{a}$, which is equivalent to $\widehat{a}b_{1}\widehat{a}^{-1}=\widehat{b}_{1}$. This equation is of the form~(E2), and since $\ell(\widehat{b}_{1})<\ell(b)<\ell(ab)=k$, we may apply the induction hypothesis. Thus there exist $\lambda\in \mathbb{F}_{2}(u,v)$ and $l\in \Z$ such that $b_1=(\lambda\widehat{\lambda})^{l}$. The fact that $b_{1}\neq 1$ implies that $\lambda\widehat{\lambda}\neq 1$ and $l\neq 0$. We claim that $\lambda\widehat{\lambda}$ may be chosen to be primitive. To prove the claim, suppose that $\lambda\widehat{\lambda}$ is not primitive. Since $\mathbb{F}_{2}(u,v)$ is a free group of rank $2$, the centraliser of $\lambda\widehat{\lambda}$ in $\mathbb{F}_{2}(u,v)$ is infinite cyclic, generated by a primitive element $v$, and replacing $v$ by $v^{-1}$ if necessary, there exists $s\geq 2$ such that $v^s= \lambda\widehat{\lambda}$. Therefore $b_{1}=v^{sl}$, and substituting this into the relation $\widehat{b}_{1}=\widehat{a}b_{1}\widehat{a}^{-1}$, we obtain $\widehat{v}^{sl}=\widehat{a} v^{sl}\widehat{a}^{-1}=(\widehat{a} v\widehat{a}^{-1})^{sl}$ in the free group $\mathbb{F}_{2}(u,v)$, from which we conclude that $\widehat{v}=\widehat{a} v\widehat{a}^{-1}$. We may thus apply the induction hypothesis to this relation because $\ell(v)<\ell(b_1)<k$, and since $v$ is primitive, there exists $\gamma\in \mathbb{F}_{2}(u,v)$ for which $v=\gamma\widehat{\gamma}$. Hence $b_{1}=(\gamma\widehat{\gamma})^{sl}$, where $\gamma\widehat{\gamma}$ is primitive, which proves the claim. Substituting $b_1=(\lambda\widehat{\lambda})^{l}$ into the relation $\widehat{a} b_1\widehat{a}^{-1}=\widehat{b}_1$, we obtain:
\begin{equation*}
(\widehat{a} \lambda\widehat{\lambda} \widehat{a}^{-1})^{l}=\widehat{a} (\lambda\widehat{\lambda})^{l}\widehat{a}^{-1}=\widehat{(\lambda\widehat {\lambda})^{l}}=(\widehat{\lambda} \lambda)^{l},
\end{equation*}
where we take $\lambda\widehat{\lambda}$ to be primitive. Once more, since $\mathbb{F}_{2}(u,v)$ is a free group of rank $2$ and $l\neq 0$, it follows that $\widehat{a} \lambda\widehat{\lambda} \widehat{a}^{-1}=\widehat{\lambda} \lambda=\widehat{\lambda} \lambda \widehat{\lambda} \widehat{\lambda}^{-1}$, from which we conclude that $\widehat{\lambda}^{-1}\widehat{a}$ belongs to the centraliser of $\lambda\widehat{\lambda}$. But $\lambda\widehat{\lambda}$ is primitive, so there exists $t\in \Z$ such that $\widehat{\lambda}^{-1}\widehat{a}=(\lambda\widehat{\lambda})^{t}$, and hence $\widehat{a}=\widehat{\lambda} (\lambda\widehat{\lambda})^{t}$. Hence $a=\lambda (\widehat\lambda \lambda)^{t}=(\lambda \widehat\lambda)^{t}\lambda$ and $b=\widehat a b_1=\widehat{\lambda} (\lambda\widehat{\lambda})^{t+l}= (\widehat{\lambda}\lambda)^{t+l}\widehat{\lambda}$ in a manner similar to that of~\reqref{lambdahat}, so~\reqref{concl1} holds.
\end{enumerate}

\item By the induction hypothesis and~(\ref{it:induct}), we may suppose that for some $k\geq 4$,~(E1) implies~\reqref{concl1} if $\ell(ab)\leq k$ and~(E2) implies~\reqref{concl2} if $\ell(w)<k$. Suppose that $\ell(w)=k$ and that $w$ and $\widehat{w}$ are conjugate. Let $\widehat{w}=\theta w\theta^{-1}$, where $\theta\in \mathbb{F}_{2}(u,v)$. If $\theta=1$ then $w=\widehat{w}$, which is impossible. So $\theta\neq 1$, and since $\ell(w)=\ell(\widehat{w})$, there must be cancellation in the expression $\theta w\theta^{-1}$. Taking the inverse of the relation $\widehat{w}=\theta w\theta^{-1}$ if necessary, we may suppose that cancellation occurs between $\theta$ and $w$.  So there exist $\theta_1,\theta_2\in \mathbb{F}_{2}(u,v)$ such that $\theta=\theta_1\theta_2$ written in reduced form, and such that the cancellation between $\theta$ and $w$ is maximal \emph{i.e.}\ if $w_{1}=\theta_{2}w$ is written in reduced form then $\theta_{1}w_{1}$ is also reduced. Let $\ell(\theta)=n$ and $\ell(\theta_2)=r$. We again consider two subcases.
\begin{enumerate}[(A)]
\item Suppose first that $r=n$. Then $\theta_{1}=1$, $\theta_{2}=\theta$ and $w=\theta^{-1}w_{1}$, so:
\begin{equation}\label{eq:ellwinequ}
\ell(w)=\ell(\theta^{-1}w_{1})\leq \ell(\theta^{-1})+\ell(w_{1})=n+\ell(w)-n=\ell(w).
\end{equation}
Hence $\ell(\theta^{-1}w_{1})= \ell(\theta^{-1})+\ell(w_{1})$, from which it follows that $w=\theta^{-1}w_1$ is written in reduced form. Therefore $\widehat{w}=\widehat{\theta}^{-1}\widehat{w}_{1}$ is also written in reduced form. Now $\widehat{w}=\theta w \theta^{-1}=w_{1}\theta^{-1}$, and applying an inequality similar to that of~\reqref{ellwinequ}, we see that $\widehat{w}=w_{1}\theta^{-1}$ is written in reduced form. Hence $\widehat{\theta}^{-1}\widehat{w}_{1}=w_{1}\theta^{-1}$, which is in the form of~(E1), both sides being written in reduced form. Thus $\ell(w_{1}\theta^{-1})= \ell(\theta^{-1}w_1)=\ell(w)=k$, and by the induction hypothesis, there exist $\lambda\in \mathbb{F}_{2}(u,v)$ and $r,s\in \Z$ such that $w_1=(\widehat{\lambda} \lambda)^s\widehat{\lambda} $ and $\theta^{-1}=(\lambda\widehat{\lambda})^r\lambda$. So by~\reqref{lambdahat}, $w=\theta^{-1}w_1= (\lambda\widehat{\lambda})^r\lambda(\widehat{\lambda} \lambda)^s\widehat{\lambda} =( \lambda \widehat{\lambda})^{r+s+1}$, which proves the result in the case $r=n$.

\item Now suppose that $r<n$. Then there must be cancellation on both sides of $w$. Taking the inverse of both sides of the equation $\widehat{w}=\theta w\theta^{-1}$ if necessary, we may suppose that the length of the cancellation on the left is less than or equal to that on the right. So there exist $\theta_1,\theta_2,\theta_3, w_2\in \mathbb{F}_{2}(u,v)$ such that $\theta=\theta_1\theta_2\theta_3$, $w=\theta_3^{-1}w_2 \theta_2\theta_3$ and $\widehat{w}=\theta_1\theta_2 w_2 \theta_1^{-1}$, all these expressions being written in reduced form. Since $\ell(w)=\ell(\widehat{w})$, it follows from the second two expressions that $\ell(\theta_1)=\ell(\theta_3)$, and that $w=\theta_3^{-1}w_2 \theta_2\theta_3=\widehat{\theta}_1\widehat{\theta}_2 \widehat{w}_2 \widehat{\theta}_1^{-1}$, written in reduced form, from which we conclude that $\widehat{\theta}_1=\theta_3^{-1}$, and that $w_2 \theta_2=\widehat{\theta}_2 \widehat{w}_2$, written in reduced form, which is in the form of~(E1). Now $\ell(w_2 \theta_2)<\ell(w)=k$, and applying the induction hypothesis, there exist $\lambda\in \mathbb{F}_{2}(u,v)$ and $r,s\in \Z$ such that $w_2=(\widehat{\lambda} \lambda)^s\widehat{\lambda} $ and $\theta_2=(\lambda\widehat{\lambda})^r\lambda$. Hence $w=\theta_3^{-1}w_2 \theta_2\theta_3= \theta_3^{-1} (\widehat{\lambda} \lambda)^{r+s+1} \theta_3=  (\theta_3^{-1}\widehat{\lambda} \widehat{\theta}_3 \ldotp\widehat{\theta}_3^{-1} \lambda \theta_3)^{r+s+1}=(\gamma \widehat{\gamma})^{r+s+1}$, where $\gamma=\theta_3^{-1}\widehat{\lambda} \widehat{\theta}_3$. This completes the proof of the induction step, and hence that of the lemma.\qedhere
\end{enumerate}
\end{enumerate}
\end{proof}
    
\begin{prop}\mbox{}\label{prop:classif}
\begin{enumerate}
\item\label{it:classifa} Let $\phi\colon\thinspace \T \multimap \T$ be a split $2$-valued map, and let $\widehat{\Phi}\colon\thinspace \T\to F_{2}(\T)$ be a lift of $\phi$ that is determined by the pair $((w^r,(a,b)), (w^s, (c,d)))$ as described in \repr{exismaps}. Let $\mathcal{O}_{\phi}$ denote the set of conjugates of this pair by elements of $B_2(\T)$. Then $\mathcal{O}_{\phi}$ is the union of the sets $\mathcal{O}_{\phi}^{(1)}$ and $\mathcal{O}_{\phi}^{(2)}$, where $\mathcal{O}_{\phi}^{(1)}$ is the subset of pairs of the form $((w_1^r,(a,b)), (w_1^s, (c,d)))$, where $w_1$ runs over the set of conjugates of $w$ in $\mathbb{F}_2(u,v)$, and $\mathcal{O}_{\phi}^{(2)}$ is the subset of pairs of the form $((w_2^r,(a+r|w|_{u},b+r|w|_{v})), (w_2^s, (c+s|w|_{u},d+s|w|_{v})))$, where $w_{2}$ runs over the set of conjugates of $\widehat{w}$ in $\mathbb{F}_2(u,v)$. Further, the correspondence that to $\phi$ associates $\mathcal{O}_{\phi}$ induces a bijection between the set of homotopy classes of split $2$-valued maps and the set of conjugates of the pairs of the form given by \repr{exismaps}(\ref{it:exismapsa}) by elements of $B_2(\T)$.

\item\label{it:classifb} Let $f=(f_1,f_2)\colon\thinspace  \T \to F_2(\T)$ be a $2$-ordered map of $\T$ determined by the pair $((w^r,(a,b)), (w^s, (c,d)))$, let $g=(g_1,g_2)\colon\thinspace  \T \to F_2(\T)$, and let $\widehat{\pi}\colon\thinspace  [\T, F_2(\T)] \to [\T, D_2(\T)]$ be the projection defined in the proof of \relem{split}(\ref{it:splitIII}). Then $\widehat{\pi}^{-1}(\widehat{\pi}([f]))= \brak{[f],[g]}$. Further, $[f]=[g]$ if and only if there exist $\lambda\in \mathbb{F}_2(u,v)$ and $l\in \Z$ such that $w=(\lambda \widehat{\lambda})^l$.
\end{enumerate}
\end{prop}

\begin{proof}\mbox{}
\begin{enumerate}[(a)]
\item To compute $\mathcal{O}_{\phi}$, we determine the conjugates of the pair $((w^r,(a,b)), (w^s, (c,d)))$ by elements of $B_2(\T)$, namely by words of the form $\sigma^{\epsilon}z$, where $\epsilon \in \brak{0,1}$, and $z\in P_2(\T)$. With respect to the decomposition of \reco{compactpres}, if $\epsilon=0$, we obtain the elements of $\mathcal{O}_{\phi}^{(1)}$. If $\epsilon=1$, using the computation for $\epsilon=0$ and the fact that 
$\sigma z^{m} \sigma^{-1}=(uv^{-1}\widehat{z}^{m} vu^{-1}, (m|z|_{u}, m|z|_{v}))$ for all $z\in \mathbb{F}_{2}(u,v)$ and $m\in \Z$ by \relem{conjhat},
we obtain the elements of $\mathcal{O}_{\phi}^{(2)}$. This proves the first part of the statement. The second part is a consequence of classical general facts about maps between spaces of type $K(\pi, 1)$, see~\cite[Chapter~V, Theorem~4.3]{Wh} for example. 

\item Let $\alpha\in [\T, F_2(\T)]$ and let $f=(f_1,f_2)\colon\thinspace  \T \to F_2(\T)$ be such that $f\in \alpha$. Taking $g=(f_2,f_1)\colon\thinspace \T \to F_2(\T)$, under the projection $\widehat{\pi}\colon\thinspace  [\T, F_2(\T)] \to [\T, D_2(\T)]$, $\widehat{\pi}(\beta)=\widehat{\pi}(\alpha)$, where $\beta=[g]$. From \resec{relnsnvm}, $\alpha$ and $\beta$ are the only elements of $[\T, F_2(\T)]$ that project under $\widehat{\pi}$ to $\widehat{\pi}(\alpha)$, which proves the first part. It remains to decide whether $\alpha=\beta$. Suppose that $f$ is determined by the pair $P_1=((w^r,(a,b)), (w^s, (c,d)))$. Since $\widehat{\pi}(\beta)=\widehat{\pi}(\alpha)$, $g$ is determined by a pair belonging to $\mathcal{O}_{\phi}$. Using the fact that $g$ is obtained from $f$ via the $\Z_2$-action on $F_2(\T)^{\T}$ that arises from the covering map $\pi$ and applying covering space arguments, there exists $g'\in \beta$ that is determined by the pair $P_2=((\widehat{w}^r,(a+r|w|_{u},b+r|w|_{v})), (\widehat{w}^s, (c+s|w|_{u},d+s|w|_{v})))$. Then $\alpha=\beta$ if and only if $P_1$ and $P_2$ are conjugate by an element of $P_2(\T)$, which is the case if and only if $w$ and $\widehat{w}$ are conjugate in the free group $\mathbb{F}_2(u,v)$ (recall from the proof of \relem{conjhat1} that if $w$ and $\widehat{w}$ are conjugate then $|w|_{u}=|w|_{v}=0$). By part~(\ref{it:classifI}) of that lemma, this is equivalent to the existence of $\lambda\in \mathbb{F}_2(u,v)$ and $l\in \Z$ such that $w=(\lambda \widehat{\lambda})^l$.\qedhere
\end{enumerate}
\end{proof}

\subsection{Fixed point theory of split $2$-valued maps}\label{sec:fptsplit2}

In this section, we give a sufficient condition for a split $2$-valued map of $\T$ to be deformable to a fixed point free $2$-valued map. \repr{lifttorus} already provides one such condition. We shall give an alternative condition in terms of roots, which seems to provide a more convenient framework from 
a computational point of view. To obtain fixed point free $2$-valued maps of $\T$, we have two possibilities at our disposal: we may either use \reth{defchineg}(\ref{it:defchinegb}), in which case we should determine the group $P_{3}(\T)$, or \reth{defchinegI}(\ref{it:defchinegbI}), in which case we may make use of the results of \resec{toro2}, and notably \repr{presTminus1alta}. We choose the second possibility. We divide the discussion into three parts. In \resec{proofexisfpf}, we prove \repr{exisfpf}. In \resec{p2Tminus1}, we give the analogue for roots of the second part of \repr{exisfpf}, and in
\resec{exrfm}, we will give some examples of split $2$-valued maps that may be deformed to root-free $2$-valued maps. 

\subsubsection{The proof of \repr{exisfpf}}\label{sec:proofexisfpf}

Let $\phi\colon\thinspace \T \multimap \T$ be a $2$-valued map of $\T$. As in \reco{compactpres}, we identify $P_{2}(\T)$ with $\mathbb{F}_{2} \times \Z^{2}$. The Abelianisation homomorphism is denoted by $\operatorname{\text{Ab}}\colon\thinspace \mathbb{F}_2(u,v) \to \Z^2$. Recall from   the beginning of \resec{relnsnvm} that if $\phi$ is split and $\widehat{\Phi}\colon\thinspace \T \to F_{2}(\T)$ is a lift of $\phi$, then $\operatorname{\text{Fix}}(\widehat{\Phi})=\operatorname{\text{Fix}}(\phi)$. In this section, we compute the Nielsen number $N(\phi)$ of $\phi$,  
and we give necessary conditions for $\phi$ to be homotopic to a fixed point free $2$-valued map. Although the Nielsen number is not the main subject of this paper, it is an important invariant in fixed point theory. The Nielsen number of $n$-valued maps was defined in~\cite{Sch1}. The following result of that paper will enable us to compute $N(\phi)$ in our setting.

\begin{thm}[{\cite[Corollary~7.2]{Sch1}}]\label{th:helgath01}
Let $n\in \N$, let $K$ be a compact polyhedron, and let $\phi=\{f_1,\ldots,f_n\}\colon\thinspace  K \multimap K$ be a split $n$-valued map. Then $N(\phi)=N(f_1)+\cdots+N(f_n)$.
\end{thm} 

\begin{proof}[Proof of \repr{exisfpf}]
Let $\phi\colon\thinspace \T \multimap \T$ be a split $2$-valued map of $\T$, and let $\widehat{\Phi}=(f_1,f_2) \colon\thinspace \T \to F_{2}(\T)$ be a lift of $\phi$ such that $\widehat{\Phi}_{\#}(e_{1})=(w^r,(a,b))$ and $\widehat{\Phi}_{\#}(e_{2})= (w^s, (c,d)))$, where $(r,s)\in \Z^{2}\setminus \brak{(0,0)}$, $a,b,c,d\in \Z$ and $w\in \mathbb{F}_2(u,v)$. For $i=1,2$, we shall compute the matrix $M_{i}$  of the homomorphism $f_{i\#}\colon\thinspace \Z^{2} \to \Z^{2}$ induced by $f_i$ on the fundamental group of $\T$ with respect to the basis $(e_{1},e_{2})$ of $\pi_1(\T)$ (up to the canonical identification of $\pi_1(\T)$ for different basepoints if necessary). In practice, the bases in the target are the images of the elements $(\rho_{1,1},\rho_{1,2})$ and $(\rho_{2,1},\rho_{2,2})$ by the homomorphism $p_{i\#}\colon\thinspace  P_2(\T) \to \pi_{1}(\T)$ induced by the projection  $p_i\colon\thinspace  F_2(\T) \to \T$ onto the $i\up{th}$ coordinate, where $i=1,2$. Note that $f_{i\#}=p_{i\#}\circ \widehat{\Phi}_{\#}$. Setting $w=w(u,v)$, and using multiplicative notation and \req{uvxy}, we have:
\begin{align*}
\widehat{\Phi}_{\#}(e_{1})&=w^r x^{a} y^{b}=(w(\rho_{1,1},\rho_{1,2}))^r (\rho_{1,1}B_{1,2}^{-1}\rho_{2,1})^{a} (\rho_{1,2}B_{1,2}^{-1}\rho_{2,2})^{b}\\
\widehat{\Phi}_{\#}(e_{2})&=w^s x^{c} y^{d}=(w(\rho_{1,1},\rho_{1,2}))^s (\rho_{1,1}B_{1,2}^{-1}\rho_{2,1})^{c} (\rho_{1,2}B_{1,2}^{-1}\rho_{2,2})^{d}.
\end{align*}
Projecting onto the first (resp.\ second) coordinate, it follows that $f_{1\#}(e_{1})=\rho_{1,1}^{rm+a}\rho_{1,2}^{rn+b}$ and $f_{1\#}(e_{2})=\rho_{1,1}^{sm+c}\rho_{1,2}^{sn+d}$ (resp.\ $f_{2\#}(e_{1})=\rho_{2,1}^{a}\rho_{2,2}^{b}$ and $f_{2\#}(e_{2})=\rho_{2,1}^{c}\rho_{2,2}^{d}$), so $M_{1} =\left( \begin{smallmatrix}
 rm+a & sm+c  \\
 rn+b & sn+d 
\end{smallmatrix}\right)$ and $M_{2} =\left( \begin{smallmatrix}
 a & c  \\
 b & d 
\end{smallmatrix}\right)$. 
One then obtains the equation for $N(\phi)$ as a consequence of \reth{helgath01} and the usual formula for the Nielsen number of a self-map of $\T$~\cite{BrooBrPaTa}. The second part of the statement is clear, since if $\phi$ can be deformed to a fixed point free $2$-valued map then $f_1, f_2$ can both be deformed to fixed point free maps. To prove the last part, $f_1$ and $f_2$ can both be deformed to fixed point free maps if and only if $\det(M_{i}-I_{2})=0$ for $i=1,2$, which using linear algebra is equivalent to:   
\begin{gather}
\text{$\det(M_{2}-I_{2})=0$, and}\label{eq:onedet}\\
\det\begin{pmatrix}
 a-1 & sm  \\
 b & sn 
\end{pmatrix}+
\det\begin{pmatrix}
 rm & c  \\
 rn & d-1 
\end{pmatrix}=0.\label{eq:twodet}
\end{gather}
Equation~\reqref{onedet} is equivalent to the proportionality of $(c,d-1)$ and $(a-1,b)$. Suppose that \req{twodet} holds. If one of the determinants in that equation is zero, then so is the other, and it follows that $(a-1, b),(c,d-1)$ and $(m,n)$ generate a subgroup of $\Z^{2}$ isomorphic to $\Z$, which yields condition~(\ref{it:exisfpfa}) of the statement. If both of these determinants are non zero then $(m,n)$ is neither proportional to $(a-1,b)$ nor to $(c,d-1)$, and since $(c,d-1)$ and $(a-1,b)$ are proportional, $(m,n)$ is not proportional to any linear combination of the two. Further,~\reqref{twodet} may be written as:
\begin{equation*}
0=s\det\begin{pmatrix}
 a-1 & m  \\
 b & n 
\end{pmatrix}+
r\det\begin{pmatrix}
 m & c  \\
 n & d-1 
\end{pmatrix}=
\det\begin{pmatrix}
s(a-1)-rc & m  \\
sb-r(d-1) & n 
\end{pmatrix},
\end{equation*}
from which it follows that $s(a-1, b)=r(c,d-1)$, which is condition~(\ref{it:exisfpfb}) of the statement. The converse is straightforward.
\end{proof}

\begin{rem}
Within the framework of \repr{exisfpf}, the fact that $f_1$ and $f_2$ can be deformed to fixed point free maps does not necessarily imply that there exists a deformation of the pair $(f_1, f_2)$, regarded as a map from $\T$ to $F_2(\mathbb{T}^{2})$, to a pair $(f_1', f_2')$ where the maps $f_1'$ and $f_2'$ are fixed point free. To answer the question of whether the $2$-ordered map $(f_1, f_2)\colon\thinspace \T \to F_2(\T)$ can be deformed or not to a fixed point free  $2$-ordered map under the hypothesis that each map can be deformed to a fixed point free map would be a major step in understanding the fixed point theory of $n$-valued maps, and would help in deciding whether the Wecken property holds or not for $\T$ for the class of split $n$-valued maps.
\end{rem}

\subsubsection{Deformations to root-free $2$-valued maps}\label{sec:p2Tminus1}

Recall from the introduction that a root of an $n$-valued map $\phi_0\colon\thinspace  X \multimap Y$, with respect  to a basepoint $y_0\in Y$, is a point $x$ such that $y_0\in \phi_0(x)$. If $\phi\colon\thinspace \T \multimap \T$ is an $n$-valued map then we may construct another $n$-valued 
map $\phi_0\colon\thinspace  \T \multimap \T$ as follows. 
If $x\in \T$ and $\phi(x)=\{ x_1,\ldots,x_n\}$, then let $\phi_0(x)=\{x_1\ldotp x^{-1},\ldots,x_n \ldotp x^{-1}\}$. The correspondence that to $\phi$ associates $\phi_{0}$ is bijective. Moreover, if $\phi$ is split, so that $\phi=\{f_1, f_2,\ldots, f_n\}$, where the self-maps $f_{i}\colon\thinspace  \T \to \T$, $i=1,\ldots,n$, are coincidence-free, then $\phi_{0}$ is also split, and is given by $\phi_{0}(x)=\{f_1(x)\ldotp x^{-1},\ldots,  f_n(x)\ldotp x^{-1}\}$ for all $x\in \T$.  The restriction of the above-mentioned correspondence to the case where the $n$-valued maps are split is also a bijection. The following lemma implies that the question of deciding whether an $n$-valued map $\phi$ can be deformed to a fixed point free map is equivalent to deciding whether the associated map $\phi_{0}$ can be deformed to a root-free map. Let $1$ denote the basepoint of $\T$.

\begin{lem}\label{lem:equivroot}
With the above notation, a point $x_0\in \T$ is a fixed point of an $n$-valued map $\phi$ of $\T$ if and only if it is a root of the $n$-valued map $\phi_0$ (i.e.\ $1\in \phi(x_0)$). Further, $\phi$ may  be deformed to an $n$-valued map $\phi'$ such that $\phi'$ has $k$ fixed points if and only if $\phi_0$ may be deformed to an $n$-valued map $\phi_0'$ such that $\phi_0'$ has $k$ roots. 
\end{lem}

\begin{proof}
Straightforward, and left to the reader.
\end{proof}

The algebraic condition given by \reth{defchinegI}(\ref{it:defchinegbI}) is equivalent to the existence of  a homomorphism $g_{\#}\colon\thinspace \pi_1(\T) \to P_{2}(\T)$ that factors through $P_2(\T\setminus\{1\})$. The following result is the analogue for roots of the second part of \repr{exisfpf}. 

\begin{prop}\label{prop:exisrf}
Let $g\colon\thinspace \T \multimap \T$ be a split $2$-valued map, and let $\widehat{g}=(g_1, g_2)\colon\thinspace  \T \to F_{2}(\T)$ be a lift of $g$ such that $\widehat{g}_{\#}(e_{1})=(w^{r},(a',b))$ and $\widehat{g}_{\#}(e_{2})=(w^{s}, (c,d'))\in P_{2}(\T)$,  
where $(r,s)\in \Z^{2}\setminus \brak{(0,0)}$, $a',b,c,d'\in \Z$, $w\in \mathbb{F}_2(u,v)$ and $\operatorname{\text{Ab}}(w)=(m,n)$.
 If $g$ can be deformed to a root-free map, then each of the maps $g_1, g_2\colon\thinspace \T \to \T$ can be deformed to a root-free map. Further, $g_1$ and $g_2$ can both be deformed to root-free maps if and only if either:
\begin{enumerate}
\item\label{it:exisrfa}  the pairs $(a',b),(c,d')$ and $(m,n)$ belong to a cyclic subgroup of $\Z^2$, or 
\item\label{it:exisrfb} $s(a',b)=r(c,d')$.
\end{enumerate}
\end{prop}

\begin{proof}
If $g$ can be deformed to a root-free map then clearly  the maps $g_1, g_2\colon\thinspace \T \to \T$ can be deformed to root-free maps. For the second part of the statement, for $i=1,2$, consider the maps  $f_{i}\colon\thinspace  \T  \to \T$, where $f_i(x)=g_i(x)\ldotp x$, and where $g_1$ and $g_2$ can both be deformed to root-free maps. Then $f_1$ and $f_2$ can be deformed to fixed point free maps, and the maps $f_1, f_2$ are determined by the elements $(w^r, (a,b))$, $(w^s, (c,d))$ of $P_2(\T)$, where $a=a'+1$ and $d=d'+1$. By \repr{exisfpf}, either the elements $(a-1,b),(c,d-1)$ and $(m,n)$ belong to a cyclic subgroup of $\Z^2$, or $s(a-1,b)=r(c,d-1)$, which is the same as saying that either the elements $(a',b),(c,d')$ and $(m,n)$ belong to a cyclic subgroup of $\Z^2$, or $s(a',b)=r(c,d')$, and the result follows.
\end{proof}

\subsubsection{Examples of split $2$-valued maps that may be deformed to root-free $2$-valued maps}\label{sec:exrfm}

We now give a family of examples of split $2$-valued maps of $\T$ that satisfy the necessary condition of \repr{exisrf} for such a map to be deformable to a root-free map. To do so, we exhibit a family of $2$-ordered maps that we compose with the projection $\pi\colon\thinspace F_2(\T) \to D_2(\T)$ to obtain a family of split $2$-valued maps. We begin by studying a $2$-ordered map of $\T$ determined by a pair of braids of the form $((w^{r},(a,b))$, $ (w^{s}, (c,d))$, where $s(a,b)=r(c,d)$ (we make use of the notation of \repr{exisrf}). 
 
\begin{prop}\label{prop:necrootfree2}
If $\widehat{\Phi}\colon\thinspace \T \to F_2({\T})$ is a lift of a split $2$-valued map $\phi\colon\thinspace \T \multimap \T$ that satisfies $\widehat{\Phi}_{\#}(e_{1})=(w^{r},(a,b))$ and $\widehat{\Phi}_{\#}(e_{2})= (w^{s}, (c,d))$, where $w\in \mathbb{F}_2(u,v)$, $a,b,c,d\in \Z$ and $(r,s)\in \Z^{2}\setminus \brak{(0,0)}$ satisfy $s(a,b)=r(c,d)$, then $\phi$ may be deformed to a root-free  $2$-valued map. 
\end{prop}

\begin{proof} By hypothesis, the subgroup $\Gamma$ of $\Z^2$ generated by $(a,b)$ and $(c,d)$ is contained in a subgroup isomorphic to $\Z$. Let $\gamma$ be a generator of $\Gamma$. Suppose first that $r$ and $s$ are both non zero. Then we may take $\gamma=(a_0, b_0)=(\ell/r)(a,b)= (\ell/s)(c,d)$, where $\ell=\gcd(r,s)$, and the elements $(w^{r},(a,b))$ and $(w^{s}, (c,d))$ belong to the subgroup of $P_{2}(\T)$ generated by $(w^{\ell},(a_0,b_0))$. Let  $z\in P_2(\T\setminus\brak{1})$ be an element that projects to $(w^{\ell},(a_0,b_0))$ under the homomorphism $\alpha\colon\thinspace P_2(\T\setminus\brak{1})\to P_2(\T)$ induced by the inclusion $\T\setminus\brak{1} \to \T$. The map $\varphi\colon\thinspace \pi_1(\T)\to P_2(\T\setminus\brak{1})$ defined by $\varphi(e_1)=z^{r/\ell}$ and $\varphi(e_2)=z^{s/\ell}$ extends to a homomorphism, and is a lift of $\widehat{\Phi}_{\#}$. The result in this case follows by \reth{defchinegI}(\ref{it:defchinegbI}). Now suppose thar $r=0$ (resp.\ $s=0$). Then $(a,b)=(0,0)$ (resp.\ $(c,d)=(0,0)$) and $(w^{r},(a,b))$ (resp.\ $(w^{s}, (c,d))$) is trivial in $P_{2}(\T)$. Let $z\in P_2(\T\setminus\brak{1})$ be an element that projects to $(w^{s},(c,d))$ (resp.\ to $(w^{r},(a,b))$). Then we define $\varphi(e_1)=1$ and $\varphi(e_2)=z$ (resp.\ $\varphi(e_1)=z$ and $\varphi(e_2)=1$), and once more the result follows.
\end{proof}

\begin{proof}[Proof of \reth{necrootfree3}]  This follows directly from \repr{necrootfree2} and the relation between the fixed point and root problems described by \relem{equivroot}.
\end{proof}

\begin{lem}\label{lem:exfreero}
Let $k,l\in \Z$ and suppose that either $p\in \{0,1\}$ or $q\in \{0,1\}$. With the notation of \repr{presTminus1alta}, the elements $(x^py^q)^k$ and $(u^pv^q)^l$ of $P_2(\T\setminus\{1\}; (x_1,x_2))$ commute. 
\end{lem}

\begin{proof}
We will make use of \repr{presTminus1alta} and some of the relations obtained in its proof. If $p$ or $q$ is zero then the result follows easily. So it suffices to consider the two cases $p=1$ and $q=1$. By \req{conjxyuv}, for $\epsilon\in \brak{1,-1}$, we have $x^{\epsilon}v x^{-\epsilon}=u^{\epsilon}v (u^{-1}B_{1,2})^{\epsilon}$ and $y^{\epsilon}u y^{-\epsilon}=(vB_{1,2}^{-1})^{\epsilon} uv^{-\epsilon}$, and by induction on $r$, it follows that $x^{r} v
x^{-r}=u^{r} v (u^{-1}B_{1,2})^{r}$ and $y^{r}u y^{-r}=(vB_{1,2}^{-1})^{r} uv^{-r}$ for all $r\in \Z$. So if $p=1$ or $q=1$ then we have respectively:
\begin{align*}
xy^{q} uv^{q} y^{-q} x^{-1} &= x (vB_{1,2}^{-1})^{q} uv^{-q} \ldotp
v^{q}x^{-1}= uv^{q}u^{-1} u =uv^{q}, \;\text{and}\\
x^{p}y u^{p}v y^{-1} x^{-p} &= x^{p} (yvy^{-1})^{p} vx^{-p}= x^{p}
v(B_{1,2}^{-1}u)^{p} v^{-1} \ldotp v x^{-p}= u^{p} v
(u^{-1}B_{1,2})^{p}(B_{1,2}^{-1}u)^{p}=u^{p} v,
\end{align*}
as required.
\end{proof}

This enables us to prove the following proposition and \reth{construct2val}.

\begin{prop}\label{prop:construct2valprop}
Suppose that $(a, b),(c,d)$ and $(m,n)$  belong to a cyclic subgroup of $\Z^2$ generated by an element of the form $(0,q), (1,q), (p,0)$ or $(p,1)$, where $p,q\in \Z$, and let $r,s\in \Z$. Then there exist $w\in \mathbb{F}_2(u,v)$, a split $2$-valued map $\phi\colon\thinspace \T \multimap \T$ and a lift $\widehat{\Phi} \colon\thinspace  \T \to F_2(\T)$ of $\phi$ for which $\operatorname{\text{Ab}}(w)=(m,n)$,  $\widehat{\Phi}_{\#}(e_1)=((w^{r},(a,b))$ and  $\widehat{\Phi}_{\#}(e_2)= (w^{s}, (c,d))$, and such that $\phi$ can be deformed to a root-free $2$-valued map. 
\end{prop}

\begin{proof}  
Once more we apply \reth{defchinegI}(\ref{it:defchinegbI}). Let $(p,q)$ be a generator of the cyclic subgroup given in the statement. So there exist $\lambda_{1}, \lambda_{2}, \lambda_{3}\in \Z$ such that  $(a, b)=\lambda_1(p,q)$, $(c,d)=\lambda_2(p,q)$ and $(m,n)=\lambda_3(p,q)$. We define $\varphi\colon\thinspace \pi_1(\T)\to P_2(\T\setminus\brak{1})$ by $\varphi(e_1)=(u^pv^q)^{\lambda_3r}(x^py^q)^{\lambda_1}$ and $\varphi(e_2)=(u^pv^q)^{\lambda_3s}(x^py^q)^{\lambda_2}$. \relem{exfreero} implies that $\varphi$ extends to a well-defined homomorphism, and we may take $w=(u^pv^q)^{\lambda_3}$.
\end{proof}

\begin{proof}[Proof of \reth{construct2val}]  This follows directly from \repr{construct2valprop} and the relation between the fixed point and root problems described in \relem{equivroot}.
\end{proof}

\begin{rem}
\reth{construct2val} implies that there is an infinite family of homotopy classes of $2$-valued maps of $\T$ that satisfy the necessary condition of \repr{exisrf}(\ref{it:exisrfa}), and can be deformed to root-free maps.  We do not know whether there exist examples of maps that satisfy this condition but    that cannot be deformed to root-free, however it is likely that such examples exist. 
\end{rem}

\section*{Appendix: Equivalence between $n$-valued maps and maps into configuration spaces}

This appendix constitutes joint work with R.~F.~Brown. Let $n\in \N$. As observed in \resec{intro}, the set of $n$-valued functions from $X$ to $Y$ is in one-to-one correspondence with the set of functions from $X$ to $D_n(Y)$.  As we have seen in the main part of this paper, this correspondence facilitates the study of $n$-valued maps, and more specifically of their fixed point theory. In this appendix, we prove \reth{metriccont} that clarifies the topological relationship preserved by the correspondence under some mild hypotheses on $X$ and $Y$.  For the sake of completeness, we will include the proof of a simple fact (\repr{multicont}) mentioned in~\cite{Sch0} that relates the splitting of maps and the continuity of multifunctions.
   
Given a metric space $Y$, let $\mathcal K'$ be the family of non-empty compact sets of $Y$. We equip $\mathcal K'$ with the topology induced by the  Hausdorff metric on  $\mathcal K'$ defined in~\cite[Chapter~VI, Section~6]{Be}.

\begin{thm}[{\cite[Chapter~VI, Section~6, Theorem~1]{Be}}]\label{th:berge}
Let $X$ and $Y$ be metric spaces, let $\mathcal K'$ denote the family of non-empty compact sets of $Y$, let $\Gamma\colon\thinspace X \multimap Y$ be a multifunction such that for all $x\in X$, $\Gamma(x)\in \mathcal K'$ and $\Gamma(x)\ne \varnothing$. Then $\Gamma$ is continuous if and only if it is a single-valued continuous mapping from $X$ to $\mathcal K'$.
\end{thm}

As we mentioned in \resec{intro}, $F_n(Y)$ may be equipped with the topology induced by the inclusion of $F_{n}(Y)$ in $Y^n$, and $D_n(Y)$ 
may be equipped with the quotient topology using the quotient map $\pi\colon\thinspace  F_n(Y) \to D_n(Y)$, a subset $W$ of $D_n(Y)$ being open if and 
only if $\pi^{-1}(W)$ is open in $F_n(Y)$. If $Y$ is a metric space with metric $d$, the set $D_n(Y)$ is a subset of $\mathcal K'$, and the Hausdorff metric on $\mathcal K'$ mentioned above restricts to a Hausdorff metric $d_H$ on $D_n(Y)$ defined as follows. If $z,w \in D_n(Y)$ then there exist $(z_1, \ldots , z_n), (w_1, \ldots , w_n)\in F_n(Y)$ such that $z=\pi(z_1, \ldots , z_n)$ and $w=\pi(w_1, \ldots , w_n)$, and 
we define $d_H$ by:
\begin{equation*}
d_H(z,w)=\max\Bigl(\max_{1\leq i\leq n} d(z_i,w), \max_{1\leq i\leq n} d(w_i,z) \Bigr),
\end{equation*}
where $\displaystyle d(z_i,w)=\min_{1\leq j\leq n} d(z_i,w_j)$ for all $1\leq i\leq n$. Notice that $d_H(z,w)$ does not depend on the choice of representatives in $F_n(Y)$. We now prove \reth{metriccont}.
  
\begin{proof}[Proof of \reth{metriccont}] 
By \reth{berge}, it suffices to show that the set $D_n(Y)$ equipped with the Hausdorff metric $d_{H}$ is homeomorphic to the unordered configuration space $D_n(Y)$ equipped with the quotient topology, or equivalently, to show that a subset of $D_n(Y)$ is open with respect to the Hausdorff metric topology if and only if it is open with respect to the quotient topology. Let $y\in D_n(Y)$, and let $(y_1,\ldots,y_n) \in F_n(Y)$ be such that $\pi(y_1,\ldots,y_n)=y$. 
  
For the `if' part, let $U_1,\ldots, U_n$ be open balls in $Y$ whose centres are $y_1,\ldots,y_n$ respectively. Without loss of generality, we may assume  that they have the same radius $\epsilon>0$, and are pairwise disjoint. Consider the Hausdorff ball $U_{H}$ of radius $\epsilon$ in $D_n(Y)$ whose centre is $y$.
Let $z$ be an element of $U_{H}$, and let $(z_1,\ldots,z_n) \in F_n(Y)$ be such that $\pi(z_1,\ldots,z_n)=z$. Suppose that $z \notin \pi(U_1 \times \cdots \times U_n)$. We argue for a contradiction. Then there exists a ball $U_{i}$ such that 
$z_j \notin U_i$ for all $j\in \brak{1,\ldots,n}$.
So $d(y_i,z)\geq \epsilon$, and from the definition of $d_H$, it follows that $d_H(y,z)\geq \epsilon$, which contradicts the choice of $z$. Hence $z\in \pi(U_1 \times \cdots \times U_n)$, and the `if' part follows.

For the `only if' part, let us consider an open ball $U_{H}$ of radius $\epsilon>0$ in the Hausdorff metric $d_H$ whose 
centre is $y$.
We will show that there are open balls $U_1, \ldots ,U_n$ in $Y$ whose centres are $y_1, \ldots , y_n$ respectively, such that the subset of elements $z$ of $D_{n}(Y)$, where $z=\pi(z_1,\ldots,z_n)$ for some $(z_1,\ldots,z_n) \in F_n(Y)$, and where
each $U_i$ contains exactly one point of $\brak{z_1,\ldots,z_n}$, is a subset of $U_{H}$.
Define $\delta>0$ to be the minimum of $\epsilon$ and the distances $d(y_i, y_j)/2$ for all $i \ne j$.  For $j = 1, \dots , n$ let $U_{j}$ be the open ball in $Y$ of radius $\delta$ with respect to $d$ and whose centre is $y_j$.  Clearly, $U_i \cap U_j = \varnothing$ for all $i \ne j$. So for all $z=\brak{z_1, \dots , z_n}$ belonging to the set $\pi(U_1 \times \cdots \times U_n)$, each $U_j$ contains exactly one point of $\brak{z_1,\ldots,z_n}$, which up to permuting indices, we may suppose to be $z_j$. Further, for all $j = 1, \dots , n$, $d(z_j,y)=d(z_j,y_j)<\delta$ and $d(y_j,z)=d(y_j,z_j)<\delta$. 
So from the definition of $d_H$, $d_H(y,z)<\delta<\epsilon$, hence $z\in U_H$, and the `only if' part follows.
\end{proof}

Just above Lemma~1 of \cite[Section~2]{Sch0}, Schirmer wrote `Clearly a multifunction which splits into maps is continuous'. For the sake of completeness, we provide a short proof of this fact.

\begin{prop}\label{prop:multicont}
Let $n\in \N$, let $X$ be a topological space, and let $Y$ be a Hausdorff topological space. For $i=1,\ldots,n$, let $f_i\colon\thinspace X \to Y$ be continuous. Then the split $n$-valued map $\phi= \brak{f_1,\ldots,f_n}\colon\thinspace  X \multimap  Y$ is continuous.
\end{prop}

\begin{proof}
Let  $x_0 \in X$, and let $V$ be an open subset of  $Y$ such that $\phi(x_0) \cap V\ne \varnothing$. Then there exists $j\in \brak{1,\ldots,n}$ such that $f_j(x_0) \in V$. Since $f_j$ is continuous, there exists an open subset $U_j$ containing $x_0$ such that if $x \in U_j$ then $f_j(x) \in V$, so $\phi(x)\cap V\ne \varnothing$. Therefore $\phi$ is lower semi-continuous. To prove upper semi-continuity, first note that for all $x\in X$, $\phi(x)$ is closed in $Y$ because $\phi(x)$ is a finite set and $Y$ is Hausdorff. Now let  $x_0 \in X$ be such that $\phi(x_0) \subset V$. We then use the continuity of the $f_j$ to define the $U_j$ as before, and we set $U = \bigcap_{j=1}^n U_j$. Since $\phi(U) \subset V$, we have proved that $\phi$ is also upper semi-continuous, and so it is continuous.
\end{proof}

\end{document}